\documentclass[10pt,a4paper]{article}
\usepackage[utf8]{inputenc}
\usepackage[T1]{fontenc}
\usepackage{amsmath}
\usepackage{amsthm}
\usepackage{amssymb}
\usepackage{graphicx}
\usepackage{epstopdf}
\usepackage{placeins}
\usepackage{subcaption}
\newtheorem{theorem}{Theorem}[section]

\newtheorem{lemma}{Lemma}[section]
\newtheorem{example}{Example}[section]
\begin{document}

	
		
		
		\title{Bivariate Bernstein Fractal Interpolation and Numerical Integration on Triangular Domains}

		\author{APARNA M. P.$^1$ \and P. PARAMANATHAN$^2$}
		\date{
			$^1$Department of Mathematics \\ Amrita School of Physical Sciences, Coimbatore \\Amrita Vishwa Vidyapeetham, India\\mp$\_$aparna@cb.students.amrita.edu\\%
			$^2$Department of Mathematics \\ Amrita School of Physical Sciences, Coimbatore \\Amrita Vishwa Vidyapeetham, India\\p$\_$paramanathan@cb.amrita.edu\\%
		}

		\maketitle

		\begin{abstract}
			The fundamental aim of this paper is to provide the approximation and numerical integration of a discrete set of data points with Bernstein fractal approach. Using Bernstein polynomials in the iterated function system, the paper initially proposes the numerical integration formula for the data set corresponding to univariate functions. The proposed formula of integration is shown to be convergent by examining the data sets of certain weierstrass functions.
			The paper then extends the Bernstein fractal approximation and numerical integration technique to two dimensional interpolating regions. Bernstein polynomials defined over triangular domain has been used for the purpose. The triangular domain has been partitioned and the newly generated points are assigned colors in a particular manner to maintain the chromatic number as 3. Following the above mentioned construction and approximation of bivariate Bernstein fractal interpolation functions, the paper introduces the numerical double integration formula using the constructed functions. The convergence of the double integration formula towards the actual integral value of the data sets is displayed with the help of some examples including the benchmark functions. Both the newly introduced iterated function systems are verified for their hyperbolicity and the resultant fractal interpolation functions are shown to be continuous. 
		\end{abstract}
		
		
		\section{Introduction}
		Bernstein polynomials are being frequently employed today to discover the intricate mathematical representation of certain parametric curves, especially in computer graphics and related fields \cite{Gal}. These curves, also known as bezier curves, frequently display the peculiar characteristics of a fractal curve. The inherent irregularity of the fractal curves is easily captured with the Bernstein fractal approximation. Due to the Bernstein polynomials' ability to express complicated functions in terms of simple series, they find extensive use in statistics, differential equations and numerical analysis  \cite{Abdulkarim} , \cite{Doha}, \cite{Khardani}, \cite{Ahmed}, \cite{Costabile}.   Introduced by S. N. Bernstein, the theory of Bernstein polynomials has been further used to the modelling of nonlinear functions over compact sets \cite{Wen}.  The definition of generalized Bernstein polynomial and the quadrature formula using Bernstein polynomial, for equally spaced points, has been proposed in \cite{Occorsio}. Numerical differentiation technique with Bernstein polynomials is given in \cite{Costabile}.  Explicit formula for the integral of Bernstein polynomial is given in \cite{Doha}. Generalization of Bernstein polynomial approximation for the functions of several variables was carried out in \cite{Duchon}, \cite{Foupouagnigni}.\\
		
		It is always possible to find the continuous function that approximated a discrete set of data points. The closeness of the approximation varies in accordance with the irregularity of the data set involved. The application of fractal interpolation helps in minimizing the approximation error to an extend. The approximation through fractal method is better achieved by including suitable functions in the iterated function system.\\
		
		The concept of fractal interpolation functions with linear functions in the IFS was first proposed by Barnsley in 1986 \cite{Barnsley1986}. The nonlinear counter part of fractal interpolation have been carried out in \cite{Ri}, \cite{Kim} and \cite{Kobes}. The fractal interpolation to two dimension was first extended by \cite{Massopust}, where the data points on the interpolating domain were coplanar. Removing the coplanarity restriction, continuous fractal interpolation surfaces are introduced by \cite{Geronimo} with constant vertical scaling factor. Limiting the interpolation points to be collinear on the boundary of the rectangle, \cite{Dalla} considered another construction of continuous fractal interpolation surfaces. The iterated function system has been reformulated in \cite{Drakopoulos} to ensure the continuity of FIS. The authors in \cite{Ruan} added some additional conditions on the functions in the IFS to make the FIS continuous. The authors in \cite{Aparna} used vertex coloring to solve the problem of continuity. The present paper follows the approach done in \cite{Aparna}, to maintain the continuity of the FIS generated. \\
		
		The construction of fractal interpolation functions using Bernstein polynomials in the IFS is considered in \cite{Vijender}. The construction method is then extended to bivariate fractal interpolation functions over rectangular interpolating domains. \cite{Vijender} further discusses the convergence of the constructed Bernstein fractal interpolation function to the original data generating function, without imposing any condition on the vertical scaling factor. \\
		
		The initial objective of the present work is to define the numerical integration formula for the one dimensional data set, generated from univariate functions, using Bernstein polynomials in the IFS. The hyperbolicity of the IFS proposed in \cite{Vijender} has been verified and the obtained FIF is shown to be a continuous, interpolating function. The convergence of the numerical integration method is examined for the data set corresponding to some of the weierstrass functions, with their exact integral values. Secondly, the present work aims to construct bivariate Bernstein FIF for the data points defined over a triangular interpolating domain. The IFS consists of $m$th degree Bernstein polynomials. The interpolating domain is partitioned according to the partition scheme specified in \cite{Aparna}. The newly generated vertices are given colors such that the chromatic number of the partition is 3. The proposed IFS has been verified for its hyperbolicity and the continuity of the constructed FIF is proved. The construction of bivariate Bernstein FIF and the associated IFS with $m$th degree Bernstein polynomial have been illustrated in particular using the first and second degree Bernstein polynomials.  Following the construction, the paper intends to define numerical double integration formula using the newly introduced IFS. The integration formula has been further demonstrated specifically using the first and second degree Bernstein polynomials in the IFS. The derived method of integration is validated by comparing the computation results obtained with the actual integral values of the considered benchmark functions.\\		
		
		The structure of the paper is organized as follows:	Following the introduction, the second section of the paper deals with the basic definition and properties of Bernstein polynomials defined over an interval. The section also discusses the theory of barycentric coordinates, the coordinate system used for defining Bernstein polynomials over triangular regions. The third section reviews the construction of Bernstein affine FIF as proposed in \cite{Vijender} and verifies the hyperbolicity of the IFS and continuity of the FIF generated. The fourth section of the paper provides the newly formulated integration formula along with the computation results obtained. The formula of integration is verified in section five with some examples. The construction of bivariate Bernstein FIF is provided in the sixth section. The section discusses the hyperbolicity of the new IFS and the continuity of  the constructed FIF. The seventh section deals with the numerical double integration formula using bivariate Bernstein FIF. The formula of integration is shown to be accurate in the eighth section, by comparing the actual integral values of the test functions considered. By summarizing the observations, the paper concludes in section nine.

		\section{Bernstein Polynomials}
		\subsection{Bernstein Approximation of Functions Defined over Arbitrary Intervals}
		Consider the piecewise, linear interpolation function $f$ passing through the data points $\{(p_{i},q_{i}):i=1,2,...,N\}.$ Then, the $m$th  Bernstein polynomial approximation of $f$ is defined as 
		\begin{align}\label{bernGen}
			B_{m}(f,p) &= \frac{1}{(p_{N}-p_{1})^{m}}\sum_{v=0}^{m}\binom{m}{v}(p-p_{1})^{v}(p_{N}-p)^{m-v}f\Big(p_{1}+\frac{v(p_{N}-p_{1})}{m}\Big)
		\end{align}
		Then, it is easy to verify that 
		\begin{align}\label{EndbernGen}
			B_{m}(f,p_{1}) &=f(p_{1}), \,\, \text{and} \,\, B_{m}(f,p_{N})=f(p_{N}).
		\end{align}
		The first and the second degree Bernstein polynomial approximations for the function $f$ are respectively,
		\begin{align}\label{bern1}
			B_{1}(f,p) &= \frac{(q_{N}-q_{1})p}{p_{N}-p_{1}} + \frac{p_{N}q_{1}-p_{1}q_{N}}{p_{N}-p_{1}}
		\end{align}
		
		\begin{align}\label{bern2}
			B_{2}(f,p) &= c_{i}p^{2}+d_{i}p+e_{i}
		\end{align} where 
		\begin{align}\label{bern2Coef}
			\nonumber
			c_{i} &= \frac{q_{1}+q_{N}-2g(\frac{p_{1}+p_{N}}{2})}{(p_{N}-p_{1})^{2}}, \\ \nonumber d_{i} &= \frac{2(p_{1}+p_{N})g(\frac{p_{1}+p_{N}}{2})-2(p_{N}q_{1}+p_{1}q_{N})}{(p_{N}-p_{1})^{2}}, \\
			e_{i} &= \frac{p_{N}^{2}q_{1}+p_{1}^{2}q_{N}-2p_{1}p_{N}g(\frac{p_{1}+p_{N}}{2})}{(p_{N}-p_{1})^{2}}
		\end{align}

		\subsection{Bernstein Approximation of Functions Defined over Triangular Regions}
		Consider a non degenerate triangular domain $D$ with $v_{1}=(x_{1},y_{1}), v_{2}=(x_{2},y_{2}), \\ v_{3}=(x_{3},y_{3})$ where $x_{1}, x_{2}, x_{3} $ are the $x-$ coordinates of the vertices and $ y_{1}, y_{2}, y_{3} $ are the $y-$ coordinates in the cartesian coordinate system.
		Let $v=(x,y)$ be an arbitrary point in $D.$ The barycentric coordinate of $V$ with respect to $D$ is the 3-tuple $(\tau_{1},\tau_{2},\tau_{3})$ calculated as the solution of the system of linear equations 
		\begin{align*}
			v=\sum_{j=1}^{3}\tau_{j}v_{j} \,\, \text{and} \,\, \sum_{j=1}^{3} \tau_{j}=1.
		\end{align*}
		The expressions for $\tau_{1}, \tau_{2}, \tau_{3},$ after solving the system of equations, will be:
		\begin{align*}
			\tau_{1} &= \frac{\beta_{1}}{\delta}, \,\,
			\tau_{2} = \frac{\beta_{2}}{\delta}, \,\,
			\tau_{3} = \frac{\beta_{3}}{\delta} 
		\end{align*} 
		where 
		\begin{align*}
			\beta_{1} &=  \begin{vmatrix}
				x & x_{2} & x_{3}\\ 
				y & y_{2} & y_{3}\\
				1 & 1 & 1
			\end{vmatrix}, \,\, 
			\beta_{2} =  \begin{vmatrix}
				x_{1} & x & x_{3}\\ 
				y_{1} & y & y_{3}\\
				1 & 1 & 1 
			\end{vmatrix}, \,\,
			\beta_{3} =  \begin{vmatrix}
				x_{1} & x_{2} & x\\ 
				y_{1} & y_{2} & y\\ 
				1 & 1 & 1\\  
			\end{vmatrix}, \,\,
			\delta =  \begin{vmatrix}
				x_{1} & x_{2} & x_{3}\\  
				y_{1} & y_{2} & y_{3}\\  
				1 & 1 & 1
			\end{vmatrix}.
		\end{align*} It is to be noted that the barycentric coordinate of the vertices $v_{1}, v_{2},v_{3}$ with respect to $D$ are $(1,0,0), (0,1,0)$ and $(0,0,1)$ respectively.\\
		For a continuous function $f$ defined over the triangular region $D,$ the $m$th  Bernstein approximation is given by:
		\begin{align}
			B_{m}(f,v) &= \sum_{i+j+k=m} f_{i,j,k} J_{i,j,k}^{m}(v) 
		\end{align}
		where $v=(x,y)$ is an arbitrary point in $D$ with barycentric coordinate $(\tau_{1},\tau_{2},\tau_{3})$
		\begin{align*}
			f_{i,j,k} &= f\Big(\frac{i}{m}, \frac{j}{m}, \frac{k}{m}\Big) \,\, \text{and} \,\, J_{i,j,k}^{m} = \frac{m!}{i!j!k!}\tau_{1}^{i}\tau_{2}^{j}\tau_{3}^{k}
		\end{align*}
		
		If the triangle $D$ is partitioned into $N$ number of subtriangles $D_{n}$ and let $h$ be the piecewise, linear interpolation function for the data set $\{(x_{nj},y_{nj},z_{nj}):j=1,2,3 \,\,\text{and}\,\, n=1,2,...,N\}$ where $(x_{nj},y_{nj})$ denotes the cartesian coordinates of the vertices in the subtriangles and $z_{nj}$ is the value of the function at these vertices. Similar to the single variable approximation, let $B_{m}(h,x,y)$ be the $m$th  Bernstein polynomial approximation for the piecewise linear interpolation function $h.$ Then, 
		\begin{align}\label{EndB2}
			B_{m}(h,x_{j},y_{j}) &= h(x_{j},y_{j})
		\end{align} where $(x_{j},y_{j}), \,\, j=1,2,3$ are the vertices of the triangular region $D.$ \\
		When $m=1,$ the first degree Bernstein polynomial approximation of the function $h$ is
		\begin{align}\label{Bern1}
			B_{1}(h,x,y) &= E_{n}x+G_{n}y+H_{n}
		\end{align}
		with 
		\begin{align*}
			E_{n} &= \frac{z_{n1}(y_{n2}-y_{n3})+z_{n2}(y_{n3}-y_{n1})+z_{n3}(y_{n1}-y_{n2})}{x_{n1}(y_{n2}-y_{n3})-x_{n2}(y_{n1}-y_{n3})+x_{n3}(y_{n1}-y_{n2})}
		\end{align*}
		
		\begin{align*}
			G_{n} &= \frac{z_{n1}(x_{n3}-x_{n2})+z_{n2}(x_{n1}-x_{n3})+z_{n3}(x_{n2}-x_{n1})}{x_{n1}(y_{n2}-y_{n3})-x_{n2}(y_{n1}-y_{n3})+x_{n3}(y_{n1}-y_{n2})}
		\end{align*}
		
		\begin{align*}
			H_{n} &= \frac{z_{n1}(x_{n2}y_{n3}-x_{n3}y_{n2})+z_{n2}(x_{n3}y_{n1}-x_{n1}y_{n3})+z_{n3}(x_{n1}y_{n2}-x_{n2}y_{n1})}{x_{n1}(y_{n2}-y_{n3})-x_{n2}(y_{n1}-y_{n3})+x_{n3}(y_{n1}-y_{n2})}
		\end{align*}
		
		Similarly the second degree Bernstein polynomial approximation will be 
		\begin{align}\label{Bern2}
			B_{2}(h,x,y) &= K_{n}x^{2}+M_{n}y^{2}+O_{n}xy+P_{n}x+T_{n}y+U_{n}
		\end{align}
		where 
		\begin{align*}
			K_{n} &= \frac{z_{n1}(y_{n2}-y_{n3})^{2}+ z_{n2}(y_{n3}-y_{n1})^{2}+ z_{n3}(y_{n1}-y_{n2})^{2}}{(x_{n1}(y_{n2}-y_{n3})-x_{n2}(y_{n1}-y_{n3})+x_{n3}(y_{n1}-y_{n3}))^{2}} \\ &+  \frac{m_{1}(y_{n2}-y_{n3})(y_{n3}-y_{n1})+ m_{2}(y_{n2}-y_{n3})(y_{n1}-y_{n2})+  m_{3}(y_{n3}-y_{n1})(y_{n1}-y_{n2}) }{(x_{n1}(y_{n2}-y_{n3})-x_{n2}(y_{n1}-y_{n3})+x_{n3}(y_{n1}-y_{n3}))^{2}}
		\end{align*}
		
		\begin{align*}
			M_{n} &= \frac{z_{n1}(x_{n3}-x_{n2})^{2}+ z_{n2}(x_{n1}-x_{n3})^{2}+ z_{n3}(x_{n2}-x_{n1})^{2}}{(x_{n1}(y_{n2}-y_{n3})-x_{n2}(y_{n1}-y_{n3})+x_{n3}(y_{n1}-y_{n3}))^{2}} \\ &+ \frac{m_{1}(x_{n3}-x_{n2})(x_{n1}-x_{n3})+ m_{2}(x_{n3}-x_{n2})(x_{n2}-x_{n1})+  m_{3}(x_{n1}-x_{n3})(x_{n2}-x_{n1}) }{(x_{n1}(y_{n2}-y_{n3})-x_{n2}(y_{n1}-y_{n3})+x_{n3}(y_{n1}-y_{n3}))^{2}}
		\end{align*}
		
		\begin{align*}
			O_{n} &= \frac{2z_{n1}(y_{n2}-y_{n3})(x_{n3}-x_{n2})+ 2z_{n2}(y_{n3}-y_{n1})(x_{n1}-x_{n3})+ 2z_{n3}(y_{n1}-y_{n2})(x_{n2}-x_{n1})}{(x_{n1}(y_{n2}-y_{n3})-x_{n2}(y_{n1}-y_{n3})+x_{n3}(y_{n1}-y_{n3}))^{2}} \\ &+
			\frac{m_{1}\big((x_{n3}-x_{n2})(y_{n3}-y_{n1})+(y_{n2}-y_{n3})(x_{n1}-x_{n3})\big)}{(x_{n1}(y_{n2}-y_{n3})-x_{n2}(y_{n1}-y_{n3})+x_{n3}(y_{n1}-y_{n3}))^{2}} \\ &+ \frac{m_{2}\big((x_{n2}-x_{n1})(y_{n2}-y_{n3})+(y_{n1}-y_{n2})(x_{n3}-x_{n2})\big)}{(x_{n1}(y_{n2}-y_{n3})-x_{n2}(y_{n1}-y_{n3})+x_{n3}(y_{n1}-y_{n3}))^{2}} \\ &+ \frac{m_{3}\big((x_{n2}-x_{n1})(y_{n3}-y_{n1})+(y_{n1}-y_{n2})(x_{n1}-x_{n3})\big)}{(x_{n1}(y_{n2}-y_{n3})-x_{n2}(y_{n1}-y_{n3})+x_{n3}(y_{n1}-y_{n3}))^{2}} 
		\end{align*}
		
		\begin{align*}
			P_{n} &=\frac{2z_{n1}(y_{n2}-y_{n3})\big(y_{n3}x_{n2}-x_{n3}y_{n2}\big)}{(x_{n1}(y_{n2}-y_{n3})-x_{n2}(y_{n1}-y_{n3})+x_{n3}(y_{n1}-y_{n3}))^{2}} \\ &+
			\frac{2z_{n2}(y_{n3}-y_{n1})\big(y_{n1}x_{n3}-x_{n1}y_{n3}\big)}{(x_{n1}(y_{n2}-y_{n3})-x_{n2}(y_{n1}-y_{n3})+x_{n3}(y_{n1}-y_{n3}))^{2}} \\ &+
			\frac{2z_{n3}(y_{n1}-y_{n2})\big(y_{n2}x_{n1}-x_{n2}y_{n1}\big)}{(x_{n1}(y_{n2}-y_{n3})-x_{n2}(y_{n1}-y_{n3})+x_{n3}(y_{n1}-y_{n3}))^{2}} \\ &+
			\frac{m_{1}(y_{n2}-y_{n3})(y_{n1}x_{n3}-y_{n3}x_{n1})}{(x_{n1}(y_{n2}-y_{n3})-x_{n2}(y_{n1}-y_{n3})+x_{n3}(y_{n1}-y_{n3}))^{2}} \\ &+
			\frac{m_{1}(y_{n3}-y_{n1})(y_{n3}x_{n2}-y_{n2}x_{n3})}{(x_{n1}(y_{n2}-y_{n3})-x_{n2}(y_{n1}-y_{n3})+x_{n3}(y_{n1}-y_{n3}))^{2}} \\ &+
			\frac{m_{2}(y_{n2}-y_{n3})(y_{n2}x_{n1}-y_{n1}x_{n2})}{(x_{n1}(y_{n2}-y_{n3})-x_{n2}(y_{n1}-y_{n3})+x_{n3}(y_{n1}-y_{n3}))^{2}} \\ &+
			\frac{m_{2}(y_{n1}-y_{n2})(y_{n3}x_{n2}-y_{n2}x_{n3})}{(x_{n1}(y_{n2}-y_{n3})-x_{n2}(y_{n1}-y_{n3})+x_{n3}(y_{n1}-y_{n3}))^{2}} \\ &+
			\frac{m_{3}(y_{n3}-y_{n1})(y_{n2}x_{n1}-y_{n1}x_{n2})}{(x_{n1}(y_{n2}-y_{n3})-x_{n2}(y_{n1}-y_{n3})+x_{n3}(y_{n1}-y_{n3}))^{2}} \\ &+
			\frac{m_{3}(y_{n1}-y_{n2})(y_{n1}x_{n3}-y_{n3}x_{n1})}{(x_{n1}(y_{n2}-y_{n3})-x_{n2}(y_{n1}-y_{n3})+x_{n3}(y_{n1}-y_{n3}))^{2}}	
		\end{align*}
		
		\begin{align*}
			T_{n} &=\frac{2z_{n1}(x_{n3}-x_{n2})\big(y_{n3}x_{n2}-x_{n3}y_{n2}\big)}{(x_{n1}(y_{n2}-y_{n3})-x_{n2}(y_{n1}-y_{n3})+x_{n3}(y_{n1}-y_{n3}))^{2}} \\ &+
			\frac{2z_{n2}(x_{n1}-x_{n3})\big(y_{n1}x_{n3}-x_{n1}y_{n3}\big)}{(x_{n1}(y_{n2}-y_{n3})-x_{n2}(y_{n1}-y_{n3})+x_{n3}(y_{n1}-y_{n3}))^{2}} \\ &+
			\frac{2z_{n3}(x_{n2}-x_{n1})\big(y_{n2}x_{n1}-x_{n2}y_{n1}\big)}{(x_{n1}(y_{n2}-y_{n3})-x_{n2}(y_{n1}-y_{n3})+x_{n3}(y_{n1}-y_{n3}))^{2}} \\ &+
			\frac{m_{1}(x_{n3}-x_{n2})(y_{n1}x_{n3}-y_{n3}x_{n1})}{(x_{n1}(y_{n2}-y_{n3})-x_{n2}(y_{n1}-y_{n3})+x_{n3}(y_{n1}-y_{n3}))^{2}} \\ &+
			\frac{m_{1}(x_{n1}-x_{n3})(y_{n3}x_{n2}-y_{n2}x_{n3})}{(x_{n1}(y_{n2}-y_{n3})-x_{n2}(y_{n1}-y_{n3})+x_{n3}(y_{n1}-y_{n3}))^{2}} \\ &+
			\frac{m_{2}(x_{n3}-x_{n2})(y_{n2}x_{n1}-y_{n1}x_{n2})}{(x_{n1}(y_{n2}-y_{n3})-x_{n2}(y_{n1}-y_{n3})+x_{n3}(y_{n1}-y_{n3}))^{2}} \\ &+
			\frac{m_{2}(x_{n2}-x_{n1})(y_{n3}x_{n2}-y_{n2}x_{n3})}{(x_{n1}(y_{n2}-y_{n3})-x_{n2}(y_{n1}-y_{n3})+x_{n3}(y_{n1}-y_{n3}))^{2}} \\ &+
			\frac{m_{3}(x_{n1}-x_{n3})(y_{n2}x_{n1}-y_{n1}x_{n2})}{(x_{n1}(y_{n2}-y_{n3})-x_{n2}(y_{n1}-y_{n3})+x_{n3}(y_{n1}-y_{n3}))^{2}} \\ &+
			\frac{m_{3}(x_{n2}-x_{n1})(y_{n1}x_{n3}-y_{n3}x_{n1})}{(x_{n1}(y_{n2}-y_{n3})-x_{n2}(y_{n1}-y_{n3})+x_{n3}(y_{n1}-y_{n3}))^{2}}	
		\end{align*}
		
		\begin{align*}
			U_{n} &=\frac{z_{n1}(x_{n2}^{2}y_{n3}^{2}+x_{n3}^{2}y_{n2}^{2})+z_{n2}(x_{n3}^{2}y_{n1}^{2}+x_{n1}^{2}y_{n3}^{2})+z_{n3}(x_{n1}^{2}y_{n2}^{2}+x_{n2}^{2}y_{n1}^{2})}{(x_{n1}(y_{n2}-y_{n3})-x_{n2}(y_{n1}-y_{n3})+x_{n3}(y_{n1}-y_{n3}))^{2}} \\ &-
			\frac{2z_{n1}x_{n2}y_{n3}x_{n3}y_{n2}-2z_{n2}x_{n3}y_{n1}x_{n1}y_{n3}-2z_{n3}x_{n1}y_{n2}x_{n2}y_{n1}}{(x_{n1}(y_{n2}-y_{n3})-x_{n2}(y_{n1}-y_{n3})+x_{n3}(y_{n1}-y_{n3}))^{2}} \\ &+
			\frac{m_{1}\big(x_{n2}y_{n3}x_{n3}y_{n1}-x_{n2}x_{n1}y_{n3}^{2}-y_{n2}y_{n1}x_{n3}^{2}+x_{n3}y_{n2}x_{n1}y_{n3}\big)}{(x_{n1}(y_{n2}-y_{n3})-x_{n2}(y_{n1}-y_{n3})+x_{n3}(y_{n1}-y_{n3}))^{2}} \\ &+	
			\frac{m_{2}\big(x_{n2}y_{n3}x_{n1}y_{n2}-y_{n3}y_{n1}x_{n2}^{2}-x_{n3}x_{n1}y_{n2}^{2}+x_{n3}y_{n2}x_{n2}y_{n1}\big)}{(x_{n1}(y_{n2}-y_{n3})-x_{n2}(y_{n1}-y_{n3})+x_{n3}(y_{n1}-y_{n3}))^{2}} \\ &+	
			\frac{m_{3}\big(x_{n3}y_{n1}x_{n1}y_{n2}-x_{n2}x_{n3}y_{n1}^{2}-y_{n2}y_{n3}x_{n1}^{2}+x_{n1}y_{n3}x_{n2}y_{n1}\big)}{(x_{n1}(y_{n2}-y_{n3})-x_{n2}(y_{n1}-y_{n3})+x_{n3}(y_{n1}-y_{n3}))^{2}}
		\end{align*}

		\section{Construction of Univariate Fractal Interpolation Function with Bernstein Polynomial in the Iterated Function System}
		Consider the data set 
		\begin{align}\label{data1}
			\{(p_{i}, q_{i}):i=1,2,...,N, \,\, N>1\}
		\end{align} where the input arguments are ordered as $p_{1} < p_{2} < ... < p_{N}$ and the output arguments are such that $q_{i}=\psi(p_{i}), $ for $i=1,2,...,N.$ Consider the $N-1$ subintervals $I_{i}=[p_{i},p_{i+1}]$ of $I=[p_{1}, p_{N}]$  and define contractive mappings $S_{i}:I \rightarrow I_{i}$ such that 
		\begin{align}\label{EndSi}
			S_{i}(p_{1})&= p_{i}, \,\, S_{i}(p_{N})= p_{i+1}, \,\, \text{for} \,\, i=1,2,...,N-1.
		\end{align} The function $S_{i}$ satisfying this endpoint condition is given by 
		\begin{align}\label{Si}
			S_{i}(p)&=a_{i}p+b_{i}.
		\end{align}
		Set $\alpha_{i}$ as the freely chosen vertical scaling factor whose value lies in between -1 and 1, for $i=1,2,...,N-1.$ 
		Define another function $Q_{i}:I \times \mathbf{R} \rightarrow \mathbf{R},$ contractive in the second variable, such that  	
		\begin{align}\label{Qi}
			Q_{i}(p,q) &=\alpha_{i}q+V_{i}(p)
		\end{align} where $V_{i}:I \rightarrow \mathbf{R}$ defined by $$V_{i}(p)=g\circ S_{i}(p)-\alpha_{i}B_{m}(g,p)$$ for $i=1,2,...,N-1,$ $m\in \mathbf{N}.$ Note that the function $g$ is a piecewise, linear interpolation function for the data set $\{(p_{i},q_{i}):i=1,...,N\}.$ 
		Using \eqref{EndB2}, it is easier to verify that the function $Q_{i}$ is characterized by the following endpoint condition 
		\begin{align}\label{EndQi}
			Q_{i}(p_{1},q_{1}) &= q_{i}, \,\, Q_{i}(p_{N},q_{N}) = q_{i+1}, 
		\end{align} for $i=1,2,...,N-1.$ 
		The IFS will be then
		\begin{align}\label{IFSbernGen}
			w_{i}(p,q) &= (S_{i}(p),Q_{i}(p,q)) \\
			\nonumber
			&= (a_{i}p+b_{i}, \alpha_{i}q+g\circ S_{i}(p)-\alpha_{i}B_{m}(g,p)) \\
			\nonumber
			&= (a_{i}p+b_{i}, \alpha_{i}q+A_{i}p+B_{i}-\alpha_{i}B_{m}(g,p))
		\end{align} for $i=1,2,...,N-1.$
		
		
		\subsection{Iterated Function System with First Degree Bernstein Polynomial}
		When $m=1$ in \eqref{IFSbernGen}, the matrix form of the corresponding IFS becomes: 
		
		\begin{align}\label{IFSB1matrix}
			wb1_{i}(p,q)= \begin{bmatrix}
				a_{i} &  0 \\
				A_{i} & \alpha_{i}  \\
			\end{bmatrix}
			\begin{bmatrix}
				p \\q
			\end{bmatrix} +
			\begin{bmatrix}
				b_{i}\\ B_{i}-\alpha_{i}\Big (\frac{(q_{N}-q_{1})p}{p_{N}-p_{1}}+\frac{p_{N}q_{1}-p_{1}q_{N}}{p_{N}-p_{1}}\Big)
			\end{bmatrix}
		\end{align} for $i=1,2,...,N-1.$

		\subsection{Iterated Function System with Second Degree Bernstein Polynomial}
		
		Putting $m=2$ in \eqref{IFSbernGen}, the following IFS will be generated.
		\begin{align}\label{IFSB2matrix}
			wb2_{i}(p,q)= \begin{bmatrix}
				a_{i} &  0 \\
				A_{i} & \alpha_{i}  \\
			\end{bmatrix}
			\begin{bmatrix}
				p \\q
			\end{bmatrix} +
			\begin{bmatrix}
				b_{i}\\ B_{i}-\alpha_{i}\Big (c_{i}p^{2}+d_{i}p+e_{i}\Big)
			\end{bmatrix}
		\end{align} 
		where $c_{i}, \,\, d_{i} $ and $e_{i}$ are given as in \eqref{bern2Coef}, for $i=1,2,...,N-1.$

		\begin{lemma}
			Consider the data set \eqref{data1} with the IFS defined in \eqref{IFSB1matrix}, where the vertical scaling factor chosen in between 0 and 1. Then, the IFS \eqref{IFSB1matrix} is hyperbolic with respect to the metric $\sigma_{1}$ defined by $$\sigma_{1}((p_{1},q_{1}),(p_{2},q_{2})) =  |p_{1}-p_{2}|+\theta_{1}|q_{1}-q_{2}|, \,\, \text{where} \,\, $$
			$$\theta_{1}=\frac{min\{1-|a_{i}|:i=1,2...,N-1\}-\epsilon_{1}}{max\{|A_{i}|+|\alpha_{i}||\frac{q_{N}-q_{1}}{p_{N}-p_{1}}|:i=1,2,...,N-1\}}, \,\, \text{for} \,\, \epsilon_{1}>0.$$
		\end{lemma}
		
		\begin{proof}
			$\sigma_{1}\big(wb1_{i}(p_{1},q_{1}), wb1_{i}(p_{2},q_{2})\big)$
			\begin{align}
				\nonumber
				=\sigma_{1}\Bigg(\bigg(a_{i}p_{1}+b_{i}, \alpha_{i}q_{1}+A_{i}p_{1}+B_{i}-\alpha_{i}\Big(\big(\frac{q_{N}-q_{1}}{p_{N}-p_{1}}\big)p_{1}+\frac{p_{N}q_{1}-p_{1}q_{N}}{p_{N}-p_{1}}\Big)\bigg), \\ \nonumber \bigg(a_{i}p_{2}+b_{i}, \alpha_{i}q_{2}+A_{i}p_{2}+B_{i}-\alpha_{i}\Big(\big(\frac{q_{N}-q_{1}}{p_{N}-p_{1}}\big)p_{2}+\frac{p_{N}q_{1}-p_{1}q_{N}}{p_{N}-p_{1}}\Big)\bigg)\Bigg)\\ \nonumber
				\leq \bigg(|a_{i}|+\theta_{1}\Big(|A_{1}|+|\alpha_{i}|\bigg|\frac{q_{N}-q_{1}}{p_{N}-p_{1}}\bigg|\Big)\bigg)|p_{1}-p_{2}|+\theta_{1}|\alpha_{i}||q_{1}-q_{2}|
			\end{align}
			Now, choose $r_{i}=max\bigg\{|a_{i}|+\theta_{1}\Big(|A_{i}|+|\alpha_{i}|\big|\frac{q_{N}-q_{1}}{p_{N}-p_{1}}\big|\Big), |\alpha_{i}|\bigg\}$ and\\ $r=max\{r_{i}:i=1,2,...,N-1\}.$ Then, $0<r<1$ and the above expression becomes \\
			$	\bigg(|a_{i}|+\theta_{1}\Big(|A_{1}|+|\alpha_{i}|\bigg|\frac{q_{N}-q_{1}}{p_{N}-p_{1}}\bigg|\Big)\bigg)|p_{1}-p_{2}|+\theta_{1}|\alpha_{i}||q_{1}-q_{2}|$
			\begin{align*}
				&\leq r(|p_{1}-p_{2}|+\theta_{1}|q_{1}-q_{2}|) \\
				&=r \sigma_{1}((p_{1},q_{1}),(p_{2},q_{2})),
			\end{align*} which implies the IFS \eqref{IFSB1matrix} is hyperbolic.
		\end{proof}

		\begin{lemma}
			The IFS defined in \eqref{IFSB2matrix} for the data set \eqref{data1} is hyperbolic with respect to the metric $\sigma_{2}$ defined by 
			$$\sigma_{2}((p_{1},q_{1}),(p_{2},q_{2}))= |p_{1}-p_{2}|+\theta_{2}|q_{1}-q_{2}|, \,\, \text{where}$$
			$$\theta_{2}=\frac{min\{1-|a_{i}|:i=1,2...,N-1\}-\epsilon_{2}}{max\{|A_{i}|+|\alpha_{i}||d_{1}|+2|\alpha_{i}||c_{i}||p_{N}|:i=1,2,...,N-1\}}, \,\, \text{for} \,\, \epsilon_{2}>0$$ as long as the vertical scaling factor $\alpha_{i}$ lies in between -1 and 1.
		\end{lemma}
		
		\begin{proof}
			$\sigma_{2}\big(wb2_{i}(p_{1},q_{1}), wb2_{i}(p_{2},q_{2})\big)$
			\begin{align*}
				=\sigma_{2}\Bigg(\bigg(a_{i}p_{1}+b_{i}, \alpha_{i}q_{1}+A_{i}p_{1}+B_{i}-\alpha_{i}\Big(c_{i}p_{1}^{2}+d_{i}p_{1}+d_{i}e_{i}\Big)\bigg),\\
				\bigg(a_{i}p_{2}+b_{i}, \alpha_{i}q_{2}+A_{i}p_{2}+B_{i}-\alpha_{i}\Big(c_{i}p_{2}^{2}+d_{i}p_{2}+d_{i}e_{i}\Big)\bigg)\Bigg) \\
				\leq \bigg(|a_{i}|+\theta_{2}\Big(|A_{i}|+|\alpha_{i}||d_{i}|+2|\alpha_{i}||c_{i}||p_{N}|\Big)\bigg)|p_{1}-p_{2}|+\theta_{2}|\alpha_{i}||q_{1}-q_{2}|
			\end{align*}
			Now, choose $t_{i}=max\bigg\{|a_{i}|+\theta_{2}\Big(|A_{i}|+|\alpha_{i}||d_{i}|+2|d_{i}||c_{i}||p_{N}|\Big)\bigg\}$ and \\$t=max\{t_{i}:i=1,2,...,N-1\}.$ Then, $0<t<1$ and the above expression becomes\\
			$\bigg(|a_{i}|+\theta_{2}\Big(|A_{i}|+|\alpha_{i}||d_{i}|+2|\alpha_{i}||c_{i}||p_{N}|\Big)\bigg)|p_{1}-p_{2}|+\theta_{2}|\alpha_{i}||q_{1}-q_{2}|$
			\begin{align*}
				&\leq t(|p_{1}-p_{2}|+\theta_{2}|q_{1}-q_{2}|) \\
				\nonumber
				&=t \sigma_{2}((p_{1},q_{1}),(p_{2},q_{2})),
			\end{align*} which implies the IFS \eqref{IFSB2matrix} is hyperbolic.
		\end{proof}
		
		\begin{theorem}
			Let $G_{1}$ be the attractor of the hyperbolic IFS \eqref{IFSbernGen}. Then, there exists a continuous function $\psi:I \rightarrow R$ that satisfies $\psi(p_{i})=q_{i}, i=1,2,...,N$ and $G_{1}$ is the graph of $\psi.$
		\end{theorem}
		
		\begin{proof}
			Let $\mathcal{F}$ be a complete metric space as given in \cite{Barnsley1988}. Define an operator $K:\mathcal{F} \rightarrow \mathcal{F}$ by
			\begin{align}
				K(\psi)(p) &= Q_{i}(S_{i}^{-1}(p), \psi \circ (S_{i}^{-1}(p))), \,\, p \in I_{i}, \,\, i=1,2,...,N-1.
			\end{align}
			Trivially, $K$ satisfies the endpoint conditions on the space $\mathcal{F},$ for
			\begin{align*}
				(K \psi)(p_{1}) &= Q_{1}(S_{1}^{-1}(p_{1}), \psi \circ (S_{1}^{-1}(p_{1}))) \\ 
				&= Q_{1}(p_{1},q_{1}) \\
				&=q_{1}
			\end{align*}
			
			\begin{align*}
				(K \psi)(p_{N}) &= Q_{N-1}(S_{N-1}^{-1}(p_{N}), \psi \circ (S_{N-1}^{-1}(p_{N}))) \\ 
				&= Q_{N-1}(p_{N},q_{N}) \\
				&=q_{N}
			\end{align*}
			
			To prove $K$ is continuous at each of the commonly shared points, consider the point $p_{i}.$ Considering $p_{i} \in I_{i},$ 
			\begin{align*}
				(K \psi) (p_{i}) &=  Q_{i}(S_{i}^{-1}(p_{i}), \psi \circ (S_{i}^{-1}(p_{i}))) \\ 
				&= Q_{i}(p_{1},q_{1}) \\
				&=q_{1}
			\end{align*}
			Now, considering $p_{i} \in I_{i-1},$
			\begin{align*}
				(K \psi)(p_{i}) &= Q_{i-1}(S_{i-1}^{-1}(p_{i}), \psi \circ (S_{i-1}^{-1}(p_{i}))) \\ 
				&= Q_{i-1}(p_{N},q_{N}) \\
				&=q_{i}
			\end{align*} and $K$ is continuous at the interior of each of the subintervals. Therefore, $K$ is continuous. 
			The contractivity of $K$ can be proved as given in \cite{Barnsley1988}. Then, applying contraction mapping principle, $K$ has a unique fixed point $\psi$ in $\mathcal{F}.$ So, the recursive relation satisfied by the generated FIF is given by
			\begin{align}\label{recursiveb1}
				\psi(p) &= \alpha_{i}\psi(S_{i}^{-1}(p))+g(p)-\alpha_{i}B_{m}(g,S_{i}^{-1}(p))
			\end{align}
			It is trivial that $\psi$ passes through the interpolation points.
			Finally, as given in \cite{Barnsley1988}, it is easier to prove that $G_{1}$ is the graph of the FIF generated. 	
		\end{proof}

		%
		
		\section{Numerical Integration Using Bernstein Polynomial in the Iterated Function System}
		
		Let $M_{1}$ denotes the integral of the fractal interpolation function over the interval $I,$ generated with $m$th Bernstein polynomial in the IFS. Then,
		\begin{align*}
			M_{1} &= \sum_{i=1}^{N-1} \int_{p_{i}}^{p_{i+1}}\psi(p) dp 
		\end{align*}
		Using the recursive relation \eqref{recursiveb1}, 
		\begin{align*}
			M_{1} &= \sum_{i=1}^{N-1} \int_{p_{i}}^{p_{i+1}} \alpha_{i}\psi(S_{i}^{-1}(p))+g(p)-\alpha_{i}B_{m}(g,S_{i}^{-1}(p)) dp
		\end{align*}
		Taking $p^{'}=S_{i}^{-1}(p), $ the above integral becomes,
		\begin{align*}
			M_{1} &= \sum_{i=1}^{N-1} \int_{p_{1}}^{p_{N}} \alpha_{i}a_{i}\psi(p^{'})+a_{i}g(S_{i}(p^{'}))-\alpha_{i}a_{i}B_{m}(g,p^{'}) dp^{'}
		\end{align*}
		By changing the variable of integration, 
		\begin{align*}
			M_{1} &= \sum_{i=1}^{N-1} \int_{p_{1}}^{p_{N}} \alpha_{i}a_{i}\psi(p)+a_{i}g(S_{i}(p))-\alpha_{i}a_{i}B_{m}(g,p) dp \\
			\nonumber
			&=\sum_{i=1}^{N-1}\alpha_{i}a_{i}M_{1}+\sum_{i=1}^{N-1}a_{i} \int_{p_{1}}^{p_{N}} g(S_{i}(p)) dp -\sum_{i=1}^{N-1}\alpha_{i}a_{i}\int_{p_{1}}^{p_{N}} B_{m}(g,p)dp\\
			\nonumber
		\end{align*}
		which implies 
		\begin{align}\label{integral1}
			M_{1} &= \frac{\sum_{i=1}^{N-1}a_{i} \int_{p_{1}}^{p_{N}} g(S_{i}(p)) dp -\sum_{i=1}^{N-1}\alpha_{i}a_{i}\int_{p_{1}}^{p_{N}} B_{m}(g,p)dp}{1-\sum_{i=1}^{N-1}\alpha_{i}a_{i}}
		\end{align}
		
		In particular, the numerical integration formula using first degree Bernstein polynomial in the IFS is given by:
		\begin{align}
			\nonumber
			M_{1} &= \frac{(\frac{p_{N}^{2}-p_{1}^{2}}{2}) \sum_{i=1}^{N-1} a_{i}A_{i} + (p_{N}-p_{1}) \sum_{i=1}^{N-1} a_{i}B_{i}}{1-\sum_{i=1}^{N-1} a_{i}\alpha_{i}} \\& - \frac{\sum_{i=1}^{N-1} \alpha_{i}a_{i} \big(\frac{(p_{N}+p_{1})(q_{N}-q_{1})}{2}+ p_{N}q_{1}-p_{1}q_{N}\big)}{1-\sum_{i=1}^{N-1} a_{i}\alpha_{i}}
		\end{align}\label{Intbern1}
		
		Similarly, numerical integration formula using second degree Bernstein polynomial in the IFS is:
		
		\begin{align}
			\nonumber
			M_{1} &= \frac{\frac{(p_{N}^{3}-p_{1}^{3})}{3}\sum_{i=1}^{N-1} \alpha_{i}a_{i}c_{i}}{1-\sum_{i=1}^{N-1}a_{i}\alpha_{i}} + \frac{(\frac{p_{N}^{2}-p_{1}^{2}}{2})\sum_{i=1}^{N-1} a_{i}(A_{i}-\alpha_{i}d_{i})}{1-\sum_{i=1}^{N-1}a_{i}\alpha_{i}} \\& + \frac{(p_{N}-p_{1})\sum_{i=1}^{N-1}a_{i}(B_{i}-\alpha_{i}e_{i})}{1-\sum_{i=1}^{N-1}a_{i}\alpha_{i}} 
		\end{align}\label{Intbern2}

		\section{Computation Results}
		\begin{example}
			Consider the function $y=1+(3t^{3}+2t^{2}-0.7)\sum_{k=1}^{\infty}\frac{cos(10^{k}\pi t)}{2^{k}}$ over the interval $[-1, 1].$
		\end{example} 
		The actual integral value $C$ of the function is 2.00407. The difference between the numerical integral value calculated using \eqref{Intbern1} and the actual integral value reduces and the error becomes 9.0475e-05 when the interval $[-1,1]$ is divided into 31 number of subintervals. The difference between the integral value with \eqref{Intbern2} and $C$ is 0.0041 at this stage. 
		
		\begin{example}
			Consider the function $y=3t^{2}+2t+0.7-5\sum_{k=1}^{\infty} \frac{sin(6^{k}\pi t)}{2^{k}} $ where $t \in [-1,1].$ 
		\end{example}
		The actual integral value of the function is 3.4. The value of \eqref{Intbern1} is 3.4 with only 5 number of subdivisions of the interval $[-1,1].$ The difference between \eqref{Intbern2} and $C$ becomes 0.1333, at this stage. 
		
		\begin{example}
			Consider the function $y=3t^{2}+2t+0.7-5\sum_{k=1}^{\infty} \frac{sin(8^{k}\pi t)}{2^{k}} $ where $t \in [-1,1].$ 
		\end{example}
		The actual integral value of the function is 3.4. The value of \eqref{Intbern1} is 3.4 with only 5 number of subdivisions of the interval $[-1,1].$ The difference between \eqref{Intbern2} and $C$ becomes 0.1333, at this stage.\\
		The attractors obtained for the first order and second order Bernstein polynomial IFS in each of these examples are given in Figure \ref{fig1}, Figure \ref{fig2} and Figure \ref{fig3}. 
		
		\begin{figure}[th]
			\begin{subfigure}{0.5\textwidth}
				\includegraphics[width=7cm]{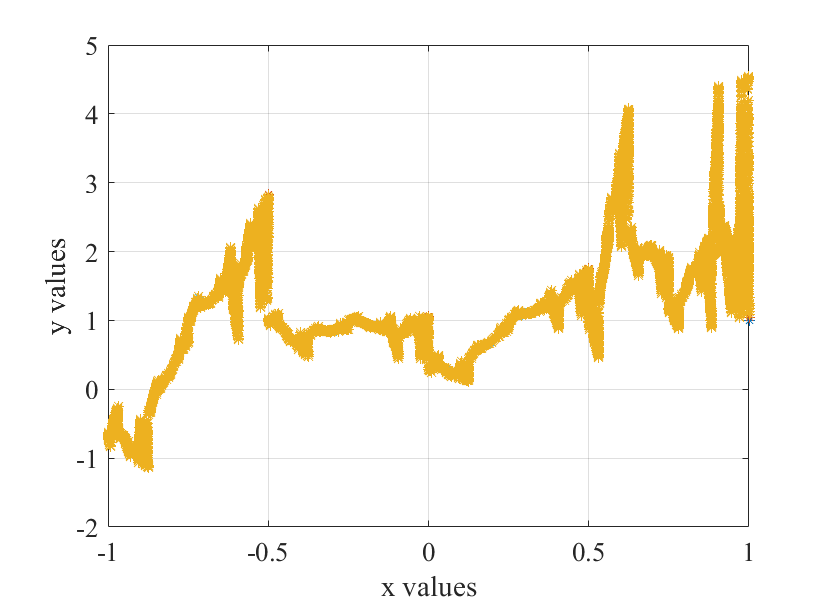}
				\vspace*{8pt}
				\caption{Attractor of \eqref{IFSB1matrix} for Example 5.1}
			\end{subfigure}
			\begin{subfigure}{0.4\textwidth}
					\includegraphics[width=7cm]{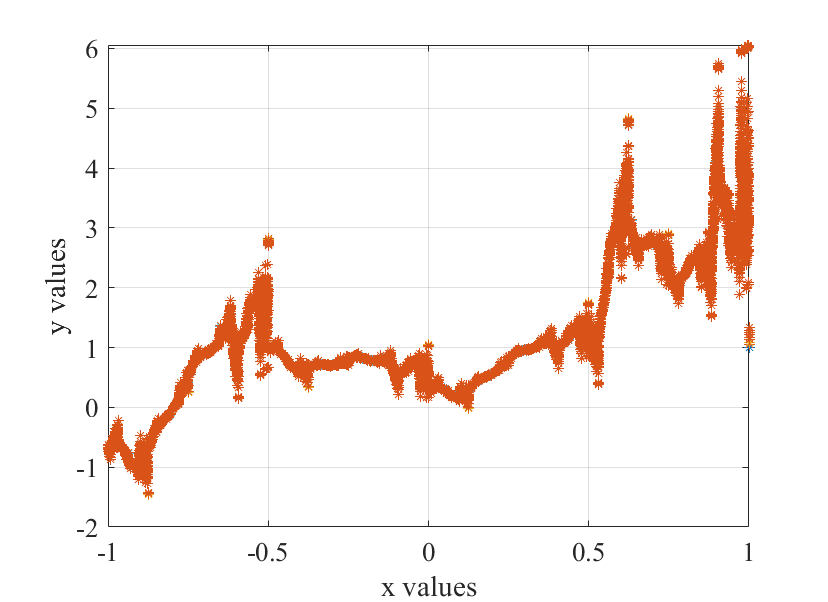}
				\vspace*{8pt}
				\caption{Attractor of \eqref{IFSB2matrix} for Example 5.1}
			\end{subfigure}
			\caption{}
			\label{fig1}
		\end{figure}
		
		\begin{figure}[th]
			\begin{subfigure}{0.5\textwidth}
				\includegraphics[width=7cm]{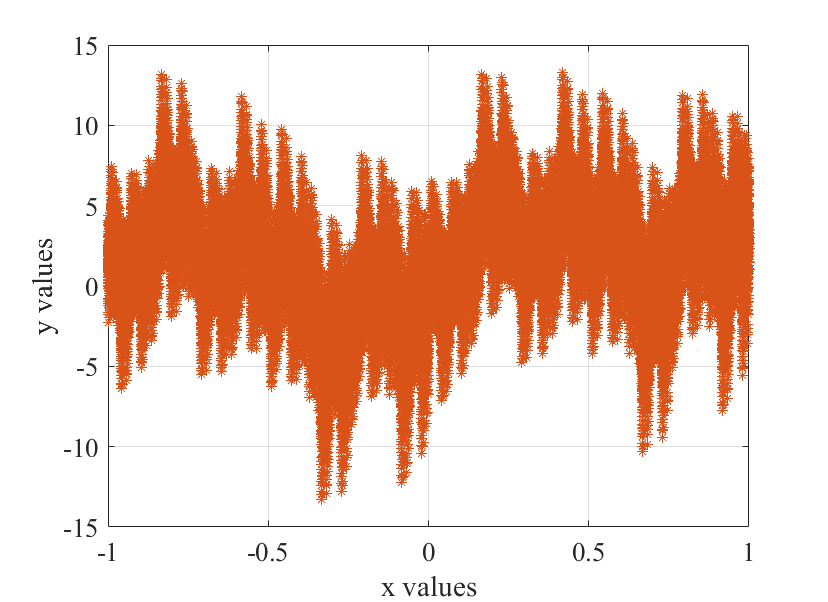}
				\vspace*{8pt}
				\caption{Attractor of \eqref{IFSB1matrix} for Example 5.2}
			\end{subfigure}
			\begin{subfigure}{0.4\textwidth}
				\includegraphics[width=7cm]{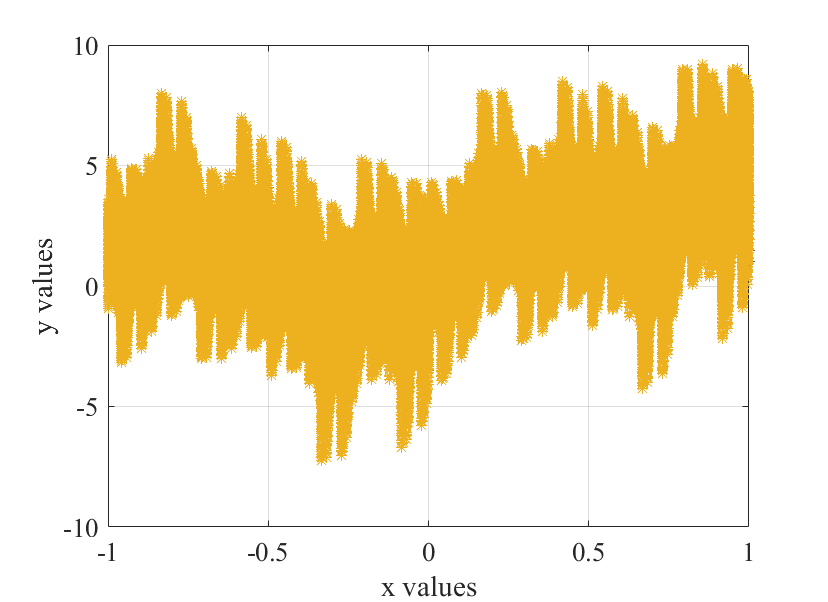}
				\vspace*{8pt}
				\caption{Attractor of \eqref{IFSB2matrix} for Example 5.2}
			\end{subfigure}
			\caption{}
			\label{fig2}
		\end{figure}
		
		\begin{figure}[th]
			\begin{subfigure}{0.5\textwidth}
				\includegraphics[width=7cm]{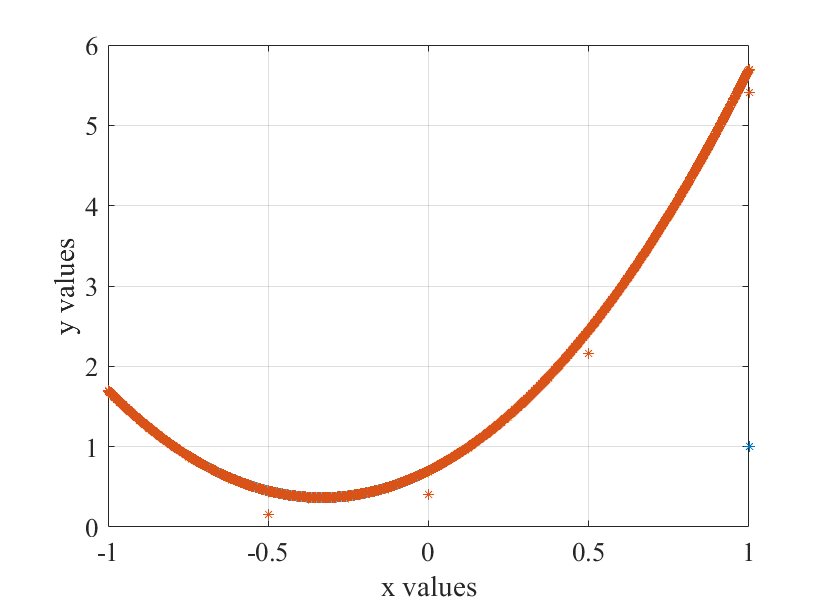}
				\vspace*{8pt}
				\caption{Attractor of \eqref{IFSB1matrix} for Example 5.3}
			\end{subfigure}
			\begin{subfigure}{0.4\textwidth}
				\includegraphics[width=7cm]{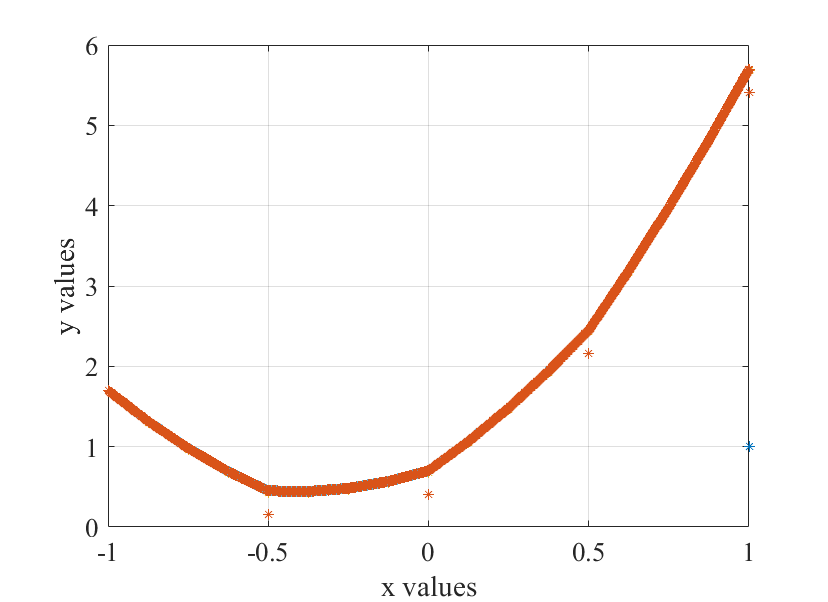}
				\vspace*{8pt}
				\caption{Attractor of \eqref{IFSB2matrix} for Example 5.3}
			\end{subfigure}
			\caption{}
			\label{fig3}
		\end{figure}

		\section{Construction of Bivariate Fractal Interpolation Functions with Bernstein Polynomials in the Iterated Function System}
		Consider a triangular region $D$ with vertices $(x_{1},y_{1}), (x_{2},y_{2}), (x_{3},y_{3}).$ Let the triangle be partitioned into $N$ number of subtriangles according to the partition scheme specified in \cite{Aparna}. Let each of the vertices be assigned with colors as given in Figure 6 in \cite{Aparna}. For the data set $R$ consisting of the points in the partition along with their function values, consider the invertible affine mapping $L_{n}$ as defined in \cite{Aparna}. For the set $F=D\times \mathbf{R},$ define the contractive mappings $F_{n}:F \rightarrow \mathbf{R}$ as $$F_{n}(x,y,z) = \alpha_{n7}z+Q_{n}(x,y)$$ with $$Q_{n}(x,y)= h\circ (L_{n}(x,y))-\alpha_{n7}B_{m}(h,x,y), \,\, n=1,2,...,N$$ where $h$ is the piecewise, interpolation function through the vertices of $D_{n}$ and $B_{m}(h,x,y)$ is the $m$th  Bernstein approximation of $h.$ $\alpha_{n7}$ is the freely chosen vertical scaling factor that varies in between -1 and 1. The endpoint condition of \eqref{EndB2} ensures equation no. 4 in \cite{Aparna}.
		The IFS for the data set $R$ will be then:
		\begin{align}\label{BivIFS}
			w_{n}(x,y,z) &= (L_{n}(x,y), F_{n}(x,y,z)), \text{for} \,\, n=1,2,...,N. 
		\end{align}
		In matrix notation:
		\begin{align}\label{2IFSBernGen}
			w_{n}(x,y,z)= \begin{bmatrix}
				\alpha_{n1} & \alpha_{n2} & 0 \\
				\alpha_{n3} & \alpha_{n4} & 0\\
				f_{n} & g_{n} & \alpha_{n7}  \\
			\end{bmatrix}
			\begin{bmatrix}
				x \\ y \\ z
			\end{bmatrix} +
			\begin{bmatrix}
				\beta_{n1}\\ \beta_{n2} \\ j_{n}-\alpha_{n7}B_{m}(h,x,y)
			\end{bmatrix}
		\end{align}
		
		where $h\circ L_{n}(x,y) = f_{n}x+k_{n}y+j_{n},$ for $n=1,2,...,N.$
		
		\subsection{Bivariate Iterated Function System with First Degree Bernstein Polynomials}
		Putting $m=1$ in \eqref{2IFSBernGen}, the bivariate IFS with first degree Bernstein polynomial will be:
		
		\begin{align}\label{2IFSB1matrix}
			WB1_{n}(x,y,z)= \begin{bmatrix}
				\alpha_{n1} & \alpha_{n2} & 0 \\
				\alpha_{n3} & \alpha_{n4} & 0\\
				f_{n} & k_{n} & \alpha_{n7}  \\
			\end{bmatrix}
			\begin{bmatrix}
				x \\ y \\ z
			\end{bmatrix} +
			\begin{bmatrix}
				\beta_{n1}\\ \beta_{n2} \\ j_{n}-\alpha_{n7}(E_{n}x+G_{n}y+H_{n})
			\end{bmatrix}
		\end{align} with $E_{n}, G_{n}, \,\, H_{n}$ are as given in \eqref{Bern1}, $n=1,2,...,N.$
		
		\subsection{Bivariate Iterated Function System with Second Degree Bernstein Polynomials}
		Using the second degree Bernstein polynomials, the IFS for the bivariate FIF will be:
		
		\begin{align}\label{2IFSB2matrix}
			WB2_{n}(x,y,z)= \begin{bmatrix}
				\alpha_{n1} & \alpha_{n2} & 0 \\
				\alpha_{n3} & \alpha_{n4} & 0\\
				f_{n} & k_{n} & \alpha_{n7}  \\
			\end{bmatrix}
			\begin{bmatrix}
				x \\ y \\ z
			\end{bmatrix} +
			\begin{bmatrix}
				\beta_{n1}\\ \beta_{n2} \\ j_{n}-\alpha_{n7}(K_{n}x^{2}+M_{n}y^{2}+O_{n}xy+P_{n}x+T_{n}y+U_{n})
			\end{bmatrix}
		\end{align} with $K_{n}, M_{n}, O_{n}, P_{n}, T_{n} $ and $U_{n}$ are as given in \eqref{Bern2}, $n=1,2,...,N.$
		
		\begin{lemma}
			The IFS \eqref{2IFSB1matrix} associated with the data set $R$ is hyperbolic with respect to the metric $\lambda_{1}$ defined by $$\lambda_{1}((x_{1},y_{1},z_{1}),(x_{2},y_{2},z_{2}))=|x_{1}-x_{2}|+|y_{1}-y_{2}|+\gamma_{1}|z_{1}-z_{2}|, \text{where}\,\,$$
			\begin{align*}
				\gamma_{1}=min\Big\{ \frac{min\{1-(|\alpha_{n1}|+\alpha_{n3}|):n=1,2,...,N\}-\epsilon_{3}}{max\{|f_{n}|-|\alpha_{n7}||E_{n}|:n=1,2,...,N\}},\\ 
				\frac{min\{1-(|\alpha_{n2}|+\alpha_{n4}|):n=1,2,...,N\}-\epsilon_{3}}{max\{|k_{n}|-|\alpha_{n7}||G_{n}|:n=1,2,...,N\}} \Big\}
			\end{align*} as long as the vertical scaling factor lies in between 0 and 1.
		\end{lemma}
		
		\begin{proof}
			$	\lambda_{1}(w_{n}(x_{1},y_{1},z_{1}), w_{n}(x_{2},y_{2},z_{2})) $
			\begin{align*}
				&= (\alpha_{n1}x_{1}+\alpha_{n2}y_{1}+\beta_{n1}, \alpha_{n3}x_{1}+\alpha_{n4}y_{1}+\beta_{n2}, \alpha_{n7}z_{1}+f_{n}x_{1}+k_{n}y_{1}+j_{n}\\&-\alpha_{n7}(E_{n}x_{1}+G_{n}y_{1}+H_{n}), \\
				& \hspace{0.5cm}\alpha_{n1}x_{2}+\alpha_{n2}y_{2}+\beta_{n1}, \alpha_{n3}x_{2}+\alpha_{n4}y_{2}+\beta_{n2}, \alpha_{n7}z_{2}+f_{n}x_{2}+k_{n}y_{2}+j_{n}\\&-\alpha_{n7}(E_{n}x_{2}+G_{n}y_{2}+H_{n})) \\
				&\leq [|\alpha_{n1}|+|\alpha_{n3}|+\gamma_{1}(|f_{n}|-|\alpha_{n7}||E_{n}|)]|x_{1}-x_{2}| +
				[|\alpha_{n2}|+|\alpha_{n4}|\\&+\gamma_{1}(|k_{n}|-|\alpha_{n7}||G_{n}|)]|y_{1}-y_{2}| +
				|\alpha_{n7}|\gamma_{1}|z_{1}-z_{2}|		
			\end{align*}
			Choose 
			\begin{align*}
				i_{n} &= max\{ |\alpha_{n7}|, |\alpha_{n1}|+|\alpha_{n3}|+\gamma_{1}(|f_{n}|-|\alpha_{n7}||E_{n}|), |\alpha_{n2}|+|\alpha_{n4}|+\gamma_{1}(|k_{n}|-|\alpha_{n7}||G_{n}|) \}
			\end{align*} and 
			$$i=max\{i_{n}:n=1,2,...,N\}.$$
			Then, $0<i<1$ and the above expression becomes \\
			
			\begin{align*}
				&[|\alpha_{n1}|+|\alpha_{n3}|+\gamma_{1}(|f_{n}|-|\alpha_{n7}||E_{n}|)]|x_{1}-x_{2}| +
				[|\alpha_{n2}|+|\alpha_{n4}|+\gamma_{1}(|k_{n}|-|\alpha_{n7}||G_{n}|)]|y_{1}-y_{2}| \\& \hspace{0.5cm}+
				|\alpha_{n7}|\gamma_{1}|z_{1}-z_{2}|	\\
				&\leq i [|x_{1}-x_{2}|+|y_{1}-y_{2}|+\gamma_{1}|z_{1}-z_{2}]| \\
				&\leq i\lambda_{1}((x_{1},y_{1},z_{1}),(x_{2},y_{2},z_{2}))
			\end{align*} that implies \eqref{2IFSB1matrix} is hyperbolic.
		\end{proof}
		
		\begin{lemma}
			The IFS \eqref{2IFSB2matrix} for the data set $R$ is hyperbolic with respect to the metric $\lambda_{2}$ defined by
			$$\lambda_{2}((x_{1},y_{1},z_{1}),(x_{2},y_{2},z_{2}))=|x_{1}-x_{2}|+|y_{1}-y_{2}|+\gamma_{2}|z_{1}-z_{2}|, \text{where}\,\,$$
			
			\begin{align*}
				\gamma_{2} = min\Big\{ \frac{min\{1-(|\alpha_{n1}|+\alpha_{n3}|):n=1,2,...,N\}-\epsilon_{3}}{max\{|f_{n}|+|\alpha_{n7}|[|P_{n}|+2|K_{n}||x_{2}|+|O_{n}||y_{3}|]:n=1,2,...,N\}}, \\
				\frac{min\{1-(|\alpha_{n2}|+\alpha_{n4}|):n=1,2,...,N\}-\epsilon_{3}}{max\{|k_{n}|+|\alpha_{n7}|[|T_{n}|+2|M_{n}||y_{3}|+|O_{n}||x_{2}|]:n=1,2,...,N\}}\Big\},
			\end{align*} where the vertical scaling factor is such that $0<\alpha_{n7}<1.$
		\end{lemma}

		\begin{proof}
			$	\lambda_{2}(w_{n}(x_{1},y_{1},z_{1}), w_{n}(x_{2},y_{2},z_{2})) $
			\begin{align*}
				&\leq& [|\alpha_{n1}|+|\alpha_{n3}|+\gamma_{2}(|f_{n}|+|\alpha_{n7}||P_{n}|+2|\alpha_{n7}||K_{n}||x_{2}|+|\alpha_{n7}||O_{n}||y_{3}|)] |x_{1}-x_{2}| \\
				&+& [|\alpha_{n2}|+|\alpha_{n4}|+\gamma_{2}(|k_{n}|+|\alpha_{n7}||I_{n}|+2|\alpha_{n7}||M_{n}||y_{3}|+|\alpha_{n7}||O_{n}||x_{2}|)] |y_{1}-y_{2}| \\
				&+& \hspace*{-5cm} \gamma_{2}|z_{1}-z_{2}|
			\end{align*}
			Choose 
			\begin{align*}
				l_{n}=max\{|\alpha_{n7}|, [|\alpha_{n1}|+|\alpha_{n3}|+\gamma_{2}(|f_{n}|+|\alpha_{n7}||P_{n}|+2|\alpha_{n7}||K_{n}||x_{2}|+|\alpha_{n7}||O_{n}||y_{3}|)], \\
				[|\alpha_{n2}|+|\alpha_{n4}|+\gamma_{2}(|k_{n}|+|\alpha_{n7}||I_{n}|+2|\alpha_{n7}||M_{n}||y_{3}|+|\alpha_{n7}||O_{n}||x_{2}|)] \}
			\end{align*} and 
			$$l=max\{l_{n}:n=1,2,...,N\}.$$ Then, $0<l<1.$ So, the above expression becomes 
			\begin{align*}
				&\leq l [|x_{1}-x_{2}|+|y_{1}-y_{2}|+\gamma_{1}|z_{1}-z_{2}|] \\
				&\leq l\lambda_{1}((x_{1},y_{1},z_{1}),(x_{2},y_{2},z_{2}))
			\end{align*} that implies \eqref{2IFSB2matrix} is hyperbolic.
		\end{proof}
		
		\begin{theorem}
			Consider the IFS \eqref{2IFSBernGen} with a fixed value of the vertical scaling factor, defined for the data set $R.$ Let $G_{2}$ be the attractor of this IFS. Then, $G_{2}$ is the graph of the unique, continuous function that passes through the given data set $R.$ 
		\end{theorem}
		
		\begin{proof}
			Let $\mathbb{F}$ be a complete metric space as defined in \cite{Aparna}. As specified in \cite{Aparna}, it is easier to show that $L_{n}^{-1}$ is well defined along each common boundary edges of the subtriangles. Now, the definition of $T$ and the fixed value of the vertical scaling factor ensures the well definiteness of $T.$ Hence, $T$ is continuous everywhere. Clearly, $T$ satisfies the endpoint conditions and the contractivity conditions \cite{Aparna}. Hence, as the unique fixed point of $T,$ the recursive relation satisfied by the bivariate FIF is 
			\begin{align}\label{recursiveB}
				f(x,y) &=\alpha_{n7}f(L_{n}^{-1}(x,y))+h(x,y)-\alpha_{n7}B_{m}(h,L_{n}^{-1}(x,y))
			\end{align}	
			Finally, it is easier to establish that $G_{2}$ is the graph of the function $f.$
		\end{proof}
		
		\section{Numerical Double Integration using Bernstein Polynomial in the Iterated Function System}
		Let $M_{2}$ denotes the numerical double integration of bivariate FIF over the triangular region $D,$ constructed with $m$th Bernstein polynomial in the IFS. Then,
		
		\begin{align*}
			M_{2} &= \sum_{n=1}^{N} \iint \limits_{D_{n}} f(x,y) dxdy
		\end{align*}
		With the recursive relation \eqref{recursiveB},
		the double integral becomes,
		\begin{align*}
			M_{2} &=\sum_{n=1}^{N} \iint\limits_{D_{n}} \alpha_{n7}f(L_{n}^{-1}(x,y))+h(x,y)-\alpha_{n7}B_{m}(h,L_{n}^{-1}(x,y))dxdy
		\end{align*}
		putting $(u,v) = L_{n}^{-1}(x,y),$
		\begin{align*}
			M_{2} &= \sum_{n=1}^{N} \iint\limits_{D}[\delta_{n}\alpha_{n7}f(u,v) + \delta_{n}h(L_{n}(u,v)) -\delta_{n}\alpha_{n7}B_{m}(h,u,v)]dudv
		\end{align*}
		changing variable of integration
		\begin{align*}
			M_{2}&=\sum_{n=1}^{N} \iint\limits_{D}[\delta_{n}\alpha_{n7}f(x,y) + \delta_{n}h(L_{n}(x,y)) -\delta_{n}\alpha_{n7}B_{m}(h,x,y)]dxdy\\
			&= \sum_{n=1}^{N}\alpha_{n7}\delta_{n}M_{2}+ \sum_{n=1}^{N}\delta_{n}\iint\limits_{D} h(L_{n}(x,y))dxdy-\sum_{n=1}^{N}\delta_{n}\alpha_{n7}\iint\limits_{D}B_{m}(h,x,y)dxdy
		\end{align*}
		which implies
		\begin{align}
			M_{2} &= \frac{\sum_{n=1}^{N} \delta_{n} \iint \limits_{D} h(L_{n}(u,v)) dudv - \sum_{n=1}^{N} \delta_{n} \alpha_{n7} \iint\limits_{D} B_{m}(h,u,v) dudv}{1-\sum_{n=1}^{N}\delta_{n}\alpha_{n7}}
		\end{align}
		
		With the first degree Bernstein polynomials in the IFS, the integral formula becomes,
		\begin{align}
			M_{2} &= \frac{\sum_{n=1}^{N} \delta_{n} \iint\limits_{D} [(f_{n}-\alpha_{n7}E_{n})x +(k_{n}-\alpha_{n7}G_{n})y + (j_{n}-\alpha_{n7}H_{n})] dxdy}{1-\sum_{n=1}^{N}\delta_{n}\alpha_{n7}}
		\end{align}
		Using the second degree Bernstein polynomials in the IFS, this becomes,
		\begin{align}
			M_{2} &= \frac{ \sum_{n=1}^{N}\delta_{n} \iint \limits_{D} [(f_{n}-\alpha_{n7}P_{n})x+(k_{n}-\alpha_{n7}T_{n})y+(j_{n}-\alpha_{n7}U_{n})}{1-\sum_{n=1}^{N}\delta_{n}\alpha_{n7}}\\ &-\frac{ \sum_{n=1}^{N}\delta_{n} \iint \limits_{D}[\alpha_{n7}K_{n}x^{2}-\alpha_{n7}M_{n}y^{2}-\alpha_{n7}O_{n}xy]dxdy}{1-\sum_{n=1}^{N}\delta_{n}\alpha_{n7}}
		\end{align}
		
		\section{Computation Results}
		\begin{example}
			Consider the function $f(x,y)= ye^{2.5x+0.6y}$ where $(x,y) $ lies in the triangular domain with vertices $(0,0), (1,0)$ and $(0.5,1).$
		\end{example}	
		Suppose this domain is partitioned into 27 number of subtriangles and each vertices are colored as  defined in \cite{Aparna}, the following data set will be generated:
		$$
		\big\{(0,0), (0.25,0), (0.5,0), (0.75,0), (1,0), (0.12, 0.25),$$$$ \hspace{1.2cm}(0.31, 0.25), (0.5, 0.25), (0.68,0.25), (0.87, 0.25),	(0.25,0.5), $$$$ \hspace{1.3cm}(0.375, 0.5), (0.5,0.5), (0.625, 0.5), (0.75,0.5),
		(0.375,0.75),$$$$\hspace{1.3cm} (0.43, 0.75), (0.5,0.75), (0.56, 0.75), (0.625, 0.75),
		(0.5, 1)\big\}
		$$
		Then, the functions in the IFS \eqref{2IFSB1matrix} will be:
		$$ \hspace{-5cm}
		\big\{(0.25x,0.25y,0.001z+0.39y  ), 
		...,$$$$
		(-0.125x+0.52, 0.25y+0.75, 0.001z-1.29x+2.19y+4.80)\big\}
		$$
		
		The IFS \eqref{2IFSB2matrix} will be:
		$$\hspace{-3.4cm}
		\big\{(0.25x, 0.25y, 0.001z+0.3964y-0.0034y^{2}-0.0009xy),...,$$$$
		(-0.125x+0.56, 0.25y+0.75, 0.001z-1.04x+2.3y+4.69-0.26x^{2}-0.062y^{2}+0.01xy\big\}	$$
		
		The integral results obtained with these IFS are tabulated in Table \ref{tab1} and the attractors are plotted in Figure \ref{fig4}.
		
		\begin{table}[th]
			\caption{Comparison table for Example 8.1: $d$=No of subdivisions, $N$=No of subtriangles, $M_{1}$= Fractal double integral value with \eqref{2IFSB1matrix}, $M_{2}$= Fractal double integral value with \eqref{2IFSB2matrix},  $I$=Actual double integral value.
				\label{tab1}}
			{\begin{tabular}{|c|c|c|c|c|c|c|} \hline
					$d$   & $N$       & I & $M_{1}$ & $M_{1}-I$ & $M_{2}$ & $M_{2}-I$  \\ \hline
					4\hphantom{00}   &\hphantom{0} 27 &\hphantom{0} 0.8502 &\hphantom{0} 0.8465 &\hphantom{0} -0.0037 &\hphantom{0}0.8370 &\hphantom{0} -0.0132 \\ \hline
				10\hphantom{00}   &\hphantom{0} 183 &\hphantom{0} 0.8502 &\hphantom{0} 0.8490 &\hphantom{0} -0.0012 &\hphantom{0} 0.7699  &\hphantom{0} -0.0803	\\	\hline
			\end{tabular} }
		\end{table}
		
		\begin{figure}[th]
			\begin{subfigure}{0.5\textwidth}
				\includegraphics[width=7cm]{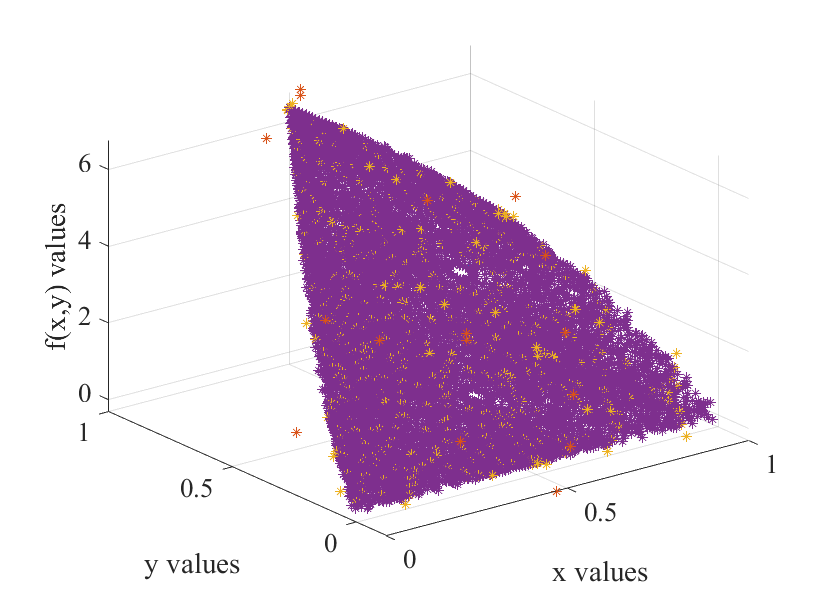}
				\vspace*{8pt}
				\caption{Attractor of \eqref{2IFSB1matrix} for Example 8.1}
			\end{subfigure}
			\begin{subfigure}{0.4\textwidth}
				\includegraphics[width=7cm]{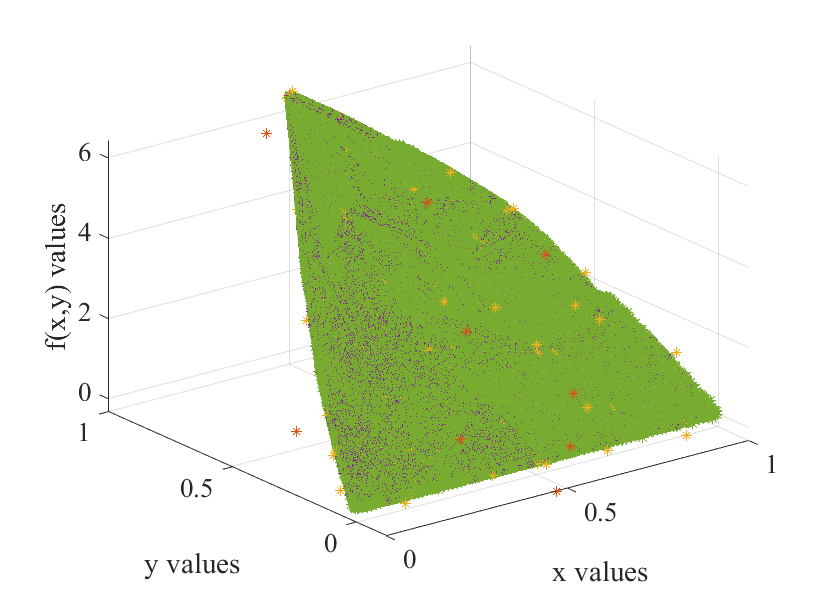}
				\vspace*{8pt}
				\caption{Attractor of \eqref{2IFSB2matrix} for Example 8.1}
			\end{subfigure}
			\caption{}
			\label{fig4}
		\end{figure}

				\begin{example}
					Consider the function $f(x,y)= \frac{cos(\pi x/4)+sin(\pi y/4)}{\sqrt[3]{x^{3}+y^{3}}}$ where $(x,y) $ lies in the triangular domain with vertices $(1.5,1.5), (2,1.5)$ and $(1.75,2).$
				\end{example}
				The data set for this function when the triangle is partitioned into 27 parts:
				$$
				\big\{(1.5,1.5), (1.625,1.5), (1.75,1.5), (1.875,1.5), (2,1.5), (1.56,1.625),  $$$$
				\hspace{-0.4cm}	(1.65,1.62), (1.75,1.62), (1.84,1.62), (1.93,1.62), (1.62,1.75), $$$$ 
				(1.68,1.75), (1.75,1.75), (1.81,1.75), (1.875,1.75), (1.68,1.875),$$$$  (1.71,1.875), (1.75,1.875), (1.78,1.875), (1.81,1.875) (1.75,2)\big\}	
				$$
				The IFS \eqref{2IFSB1matrix} will be:
				$$
				\big\{(0.25x+1.125, 0.25y+1.125, 0.001z-0.15x-0.01y+0.94),
				...,$$$$
				\hspace{0.3cm}	(-0.125x+1.96, 0.25y+1.5, 0.001z+0.059x-0.03y+0.46)\big\}$$
				and the IFS \eqref{2IFSB2matrix} is:
				$$
				\big\{(0.25x+1.12, 0.25y+1.125, 0.001z-3.6x-27y+47.9+11.65x^{2}+9.73y^{2}-1.99xy),...,$$$$ \hspace{0.2cm}
				(-0.125x+1.96, 0.25y+1.5, 0.001z-400.9x-84y+143.3+117.18
				x^{2}+24.13y^{2}-4.92xy)\big\}$$
				Table \ref{tab2} shows the integral results and the attractors are displayed in Figure \ref{fig5}.
				
				\begin{table}[th]
					\caption{Comparison table for Example 8.2: $d$=No of subdivisions, $N$=No of subtriangles, $M_{1}$= Fractal double integral value with \eqref{2IFSB1matrix}, $M_{2}$= Fractal double integral value with \eqref{2IFSB2matrix},  $I$=Actual double integral value.\label{tab2}
					}
					{\begin{tabular}{|c|c|c|c|c|c|c|} \hline
							$d$   & $N$       & I & $M_{1}$ & $M_{1}-I$ & $M_{2}$ & $M_{2}-I$  \\ \hline
							4\hphantom{00}   &\hphantom{0} 27 &\hphantom{0} 0.0672 &\hphantom{0} 0.0671 &\hphantom{0} -1.1093e-04 &\hphantom{0}0.3944 &\hphantom{0} 0.3272 \\
							10\hphantom{00}   &\hphantom{0} 183 &\hphantom{0} 0.0672 &\hphantom{0} 0.0672 &\hphantom{0} -7.1178e-05 &\hphantom{0}2.9716 &\hphantom{0} 2.9044 \\
							\hline
					\end{tabular} }
				\end{table}
				
				\begin{figure}[th]
					\begin{subfigure}{0.4\textwidth}
						\includegraphics[width=7cm]{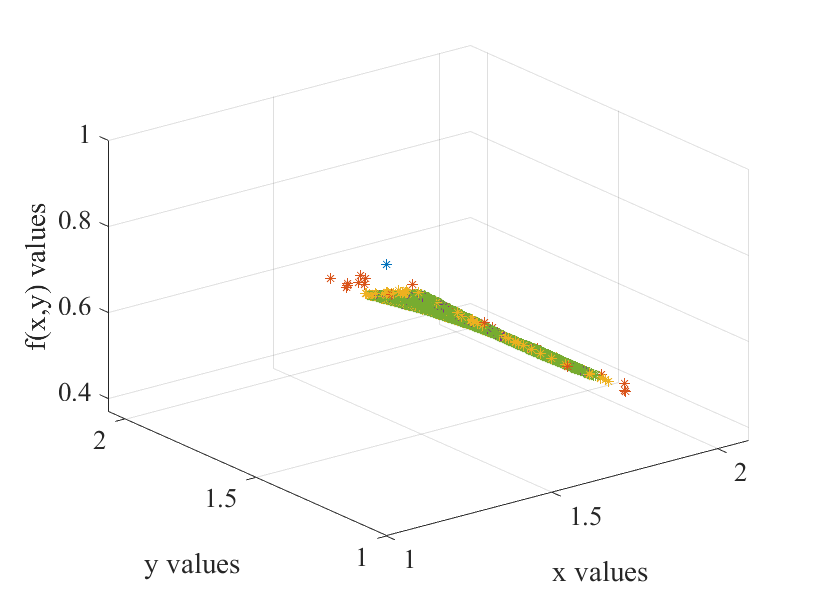}
						\vspace*{8pt}
						\caption{Attractor of \eqref{2IFSB1matrix} for Example 8.2}
					\end{subfigure}
					\begin{subfigure}{0.4\textwidth}
						\includegraphics[width=7cm]{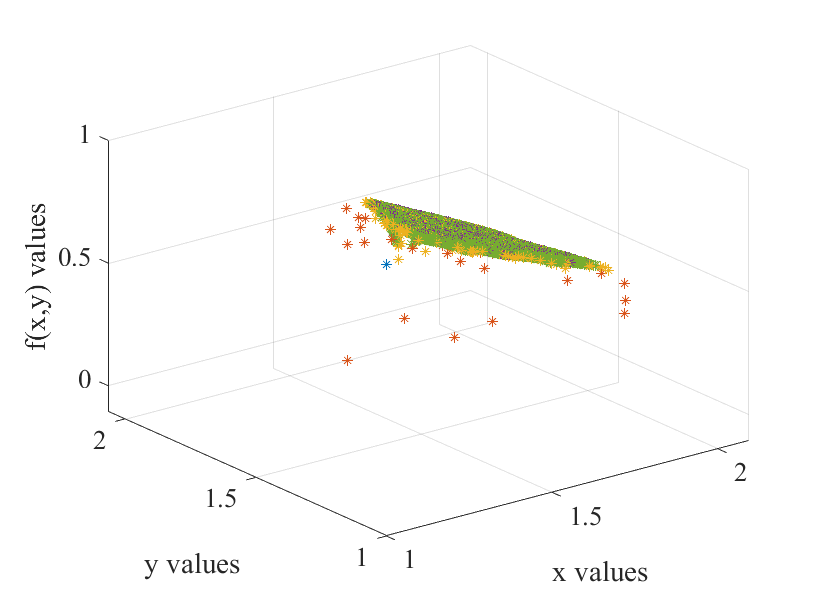}
						\vspace*{8pt}
						\caption{Attractor of \eqref{2IFSB2matrix} for Example 8.2}
					\end{subfigure}
					\caption{}
					\label{fig5}
				\end{figure}

						\begin{example}
							Consider the test function \begin{align*}
								f(x,y)=0.26(x^{2}+y^{2})-0.48xy \,\,\, 
							\end{align*} where $(x,y)$ lies in the triangular region with vertices $(-10,-10), \,\, (10,-10) $ and $(0,10).$
						\end{example}
						In accordance with the above mentioned partition scheme, the data set with 27 subtriangles in the partition is:
						$$
						\big\{(-10,-10, 4), (-5,-10,8.5),(0,-10,26),(5,-10,56.5),(10,-10,100), $$$$
						\hspace{0.8cm} (-7.5,-5,3.12),  (0,-5,6.5),    (3.75,-5,19.16),    (7.5,-5,39.13),(-5,0,6.5), $$$$ \hspace{-0.8cm}
						(-2.5,0,1.63), (0,0,0),   (2.5,0,1.63), (5,0,6.5),(-2.5,5,14.13),  $$$$
						\hspace{-0.5cm} (-1.25,5,9.91),   (0,5,6.5),    (1.25,5,3.91),    (2.5,5,2.13),	(0,10,26)\big\}
						$$		The IFS \eqref{2IFSB1matrix} is 
						$$	\big\{(0.25x-7.5, 0.25y-7.5, 0.001z+0.2241x-0.1556y+4.6807),...,$$$$ \hspace{-0.4cm}
						(-0.125x, 0.25y+7.5, 0.001z+0.3024x+0.9509y+16.4653)\big\}$$
						The IFS \eqref{2IFSB2matrix} is:
						$$\hspace{-5cm}\big\{(0.25x-7.5, 0.25y-7.5, 0.001z-0.0005x^{2}),...,$$$$
						(-0.125x, 0.25y+7.5, 0.001z+(1.0e+04)*.002x+(1.0e+04)*.002+0.0139x^{2})\big\}$$
						
						\begin{table}[th]
							\caption{Comparison table for Example 8.3: $d$=No of subdivisions, $N$=No of subtriangles, $M_{1}$= Fractal double integral value with \eqref{2IFSB1matrix}, $M_{2}$= Fractal double integral value with \eqref{2IFSB2matrix},  $I$=Actual double integral value.\label{tab3}}
							{\begin{tabular}{|c|c|c|c|c|c|c|} \hline
									$d$   & $N$       & I & $M_{1}$ & $M_{1}-I$ & $M_{2}$ & $M_{2}-I$  \\ \hline
									4\hphantom{00}   &\hphantom{0} 27 &\hphantom{0} 2600 &\hphantom{0} 3.0723e+03 &\hphantom{0} 472.2672 &\hphantom{0} 3.4592e+06 &\hphantom{0} 3.4566e+06 \\
									10\hphantom{00}   &\hphantom{0} 183 &\hphantom{0} 2600 &\hphantom{0} 2.6777e+03  &\hphantom{0} 77.7029 &\hphantom{0} 2.2021e+07 &\hphantom{0} 2.2018e+07 \\
									49\hphantom{00}   &\hphantom{0} 4707 &\hphantom{0} 2600 &\hphantom{0} 2.6047e+03  &\hphantom{0} 4.6640 &\hphantom{0} 5.5734e+08
									&\hphantom{0} 5.5733e+08
									\\
									\hline
							\end{tabular} }
						\end{table}
						
						\begin{figure}[th]
							\begin{subfigure}{0.5\textwidth}
								\includegraphics[width=7cm]{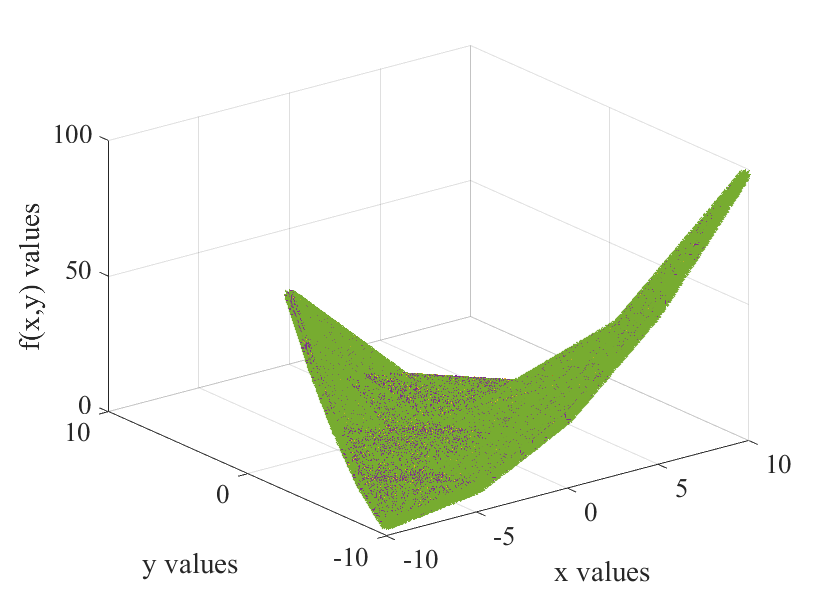}
								\vspace*{8pt}
								\caption{Attractor of \eqref{2IFSB1matrix} for Example 8.3}
							\end{subfigure}
							\begin{subfigure}{0.4\textwidth}
								\includegraphics[width=7cm]{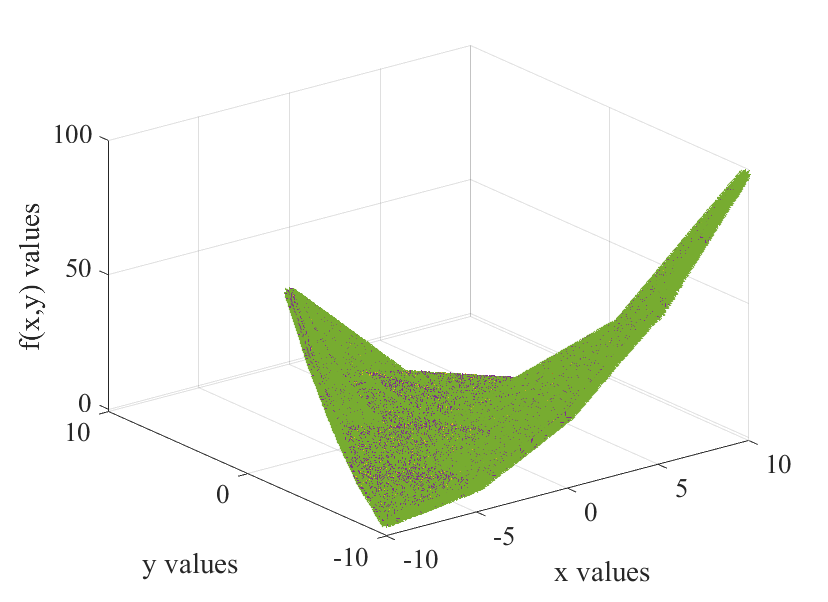}
								\vspace*{8pt}
								\caption{Attractor of \eqref{2IFSB2matrix} for Example 8.3}
							\end{subfigure}
							\caption{}
							\label{fig6}
						\end{figure}
						The obtained integral results are given in Table \ref{tab3} and the attractors are plotted in Figure \ref{fig6}. 
						
						\begin{example}
							Consider another test function \begin{align*}
								f(x,y)=(x^{2}+y-11)^{2}+(x+y^{2}-7)^{2} \,\,\, 
							\end{align*} where $(x,y)$ lies in the triangular region with vertices $(-5,-5), \,\, (5,-5) $ and $(0,5).$
						\end{example}
						In accordance with the above mentioned partition scheme, the data set with 27 subtriangles in the partition is:
						$$
						\big\{(-5,-5),(-2.5,-5),(0,-5),(2.5,-5),(5,-5),
						(-3.75,-2.5),$$$$ \hspace{0.3cm}(-1.875,-2.5), (0,-2.5), (1.875,-2.5), (3.75,-2.5),
						(-2.5,0), $$$$(-1.25,0), (0,0), (1.25,0), (2.5,0)
						(-1.25,2.5), (-0.625,2.5), $$$$\hspace{-3.2cm}(0,2.5), (0.625,2.5), (1.25,2.5),(0,5)\big\}$$
						
						The IFS \eqref{2IFSB1matrix} is 
						$$
						\big\{(0.25x-3.75, 0.25y-3.75, 0.001z+8.49x-27.1001y+156.73),
						...,$$$$
						\hspace{-1cm}(-0.125x, 0.25y+3.75, 0.001z+0.18x+29.21y+213.58)\big\}$$
						
						The IFS \eqref{2IFSB2matrix} is:
						$$
						\big\{(0.25x-3.7, 0.25y-3.7, 0.001z+7.86x-27.3y+155.19-0.09x^{2}-0.02y^{2}+0.01xy),
						...,$$$$
						\hspace{0.2cm}	(-0.125x, 0.25y+3.75, 0.001z+0.23x+29.61y+212.94-0.07x^{2}-0.05y^{2}-0.02xy)\big\}$$

						\begin{table}[th]
							\caption{Comparison table for Example 8.4: $d$=No of subdivisions, $N$=No of subtriangles, $M_{1}$= Fractal double integral value with \eqref{2IFSB1matrix}, $M_{2}$= Fractal double integral value with \eqref{2IFSB2matrix},  $I$=Actual double integral value.\label{tab4}}
							{\begin{tabular}{|c|c|c|c|c|c|c|} \hline
									$d$   & $N$       & I & $M_{1}$ & $M_{1}-I$ & $M_{2}$ & $M_{2}-I$  \\ \hline
									4\hphantom{00}   &\hphantom{0} 27 &\hphantom{0} 7625 &\hphantom{0} 9.8090e+03 &\hphantom{0} 2.1840e+03 &\hphantom{0} 9.6395e+03 &\hphantom{0} 2.0145e+03  \\
									10\hphantom{00}   &\hphantom{0} 183 &\hphantom{0} 7625 &\hphantom{0}  7.9803e+03  &\hphantom{0} 355.2784 &\hphantom{0} 7.2675e+03 &\hphantom{0} -357.4685 \\
									\hline
							\end{tabular} }
						\end{table}
						
						\begin{figure}[th]
							\begin{subfigure}{0.5\textwidth}
								\includegraphics[width=7cm]{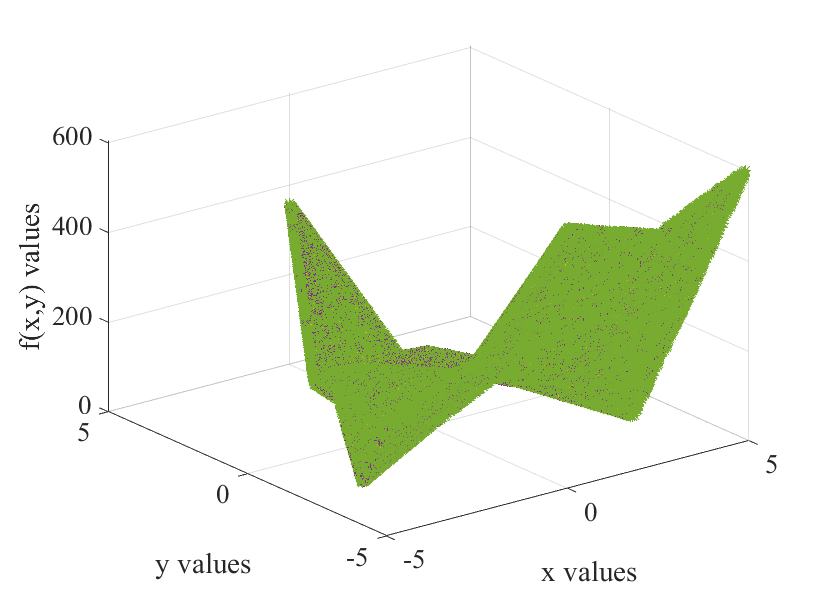}
								\vspace*{8pt}
								\caption{Attractor of \eqref{2IFSB1matrix} for Example 8.4}
							\end{subfigure}
							\begin{subfigure}{0.4\textwidth}
								\includegraphics[width=7cm]{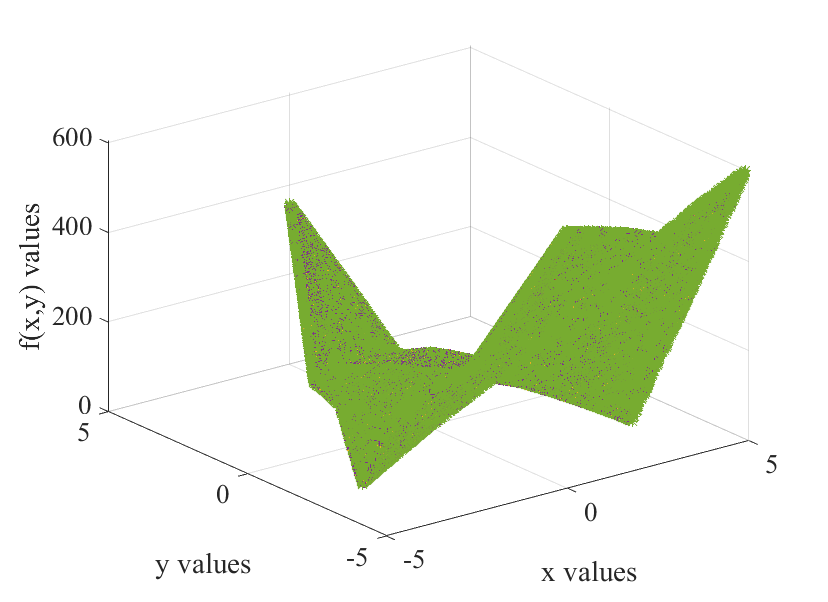}
								\vspace*{8pt}
								\caption{Attractor of \eqref{2IFSB2matrix} for Example 8.4}
							\end{subfigure}
							\caption{}
							\label{fig7}
						\end{figure}
						
						The integral values are provided in Table \ref{tab4} with the attractors in Figure \ref{fig7}.
						It is to be noted that the integral values are converging towards 0 as the number of subtriangles is increased. However, for the above two test functions (Examples 8.3 and 8.4), the actual convergence to 0 can be verified with the help of high performing computers.

						\section{Conclusion}
						With the help of Bernstein polynomials, this paper introduces the approximation of bivariate fractal interpolation functions over triangular interpolating domain. The considered interpolating domain is partitined in a specific manner and the vertices are given colors such that the chromatic number of the partition is 3. The contractive mappings in the IFS consists of $m$th Bernstein polynomials. The coefficients of the IFS are used to formulate the numerical double integration formula for the bivariate signals. The numerical integration of univariate signals is carried out using the coefficients of the Bernstein polynomial IFS defined over closed intervals. The accuracy of both the integration formulae is verified through a couple of test functions. Analysing the results obtained, it has been observed that the IFS with the first degree Bernstein polynomial exhibits better converging tendency as compared with the higher degree Bernstein polynomials.


\begin{thebibliography}{99}
							\bibitem{Barnsley1986} Barnsley, M. F. (1986). Fractal functions and interpolation. {\it Constr. Approx.}, {\bf 2}: 303--329.
							\bibitem{Barnsley1988} 
							Barnsley, M. F. (1988). {\it Fractals Everywhere}, Academic Press, New York.
							\bibitem{Gal} Gal, S. G. (2010). {\it Shape-preserving approximation by real and complex polynomials}, Springer Science \& Business Media.
							\bibitem{Abdulkarim} Abdul Karim, S. A., Khan, F. and Basit, M. (2022). Symmetric Bernstein polynomial approach for the system of Volterra integral equations on arbitrary interval and its convergence analysis. {\it Symmetry}, {\bf 14}: 1343.
							\bibitem{Doha} Doha, E. H., Bhrawy, A. H. and Saker, M. A. (2011). Integrals of Bernstein polynomials: An application for the solution of high even-order differential equations. {\it Applied Mathematics Letters}, {\bf 24}: 559--565.
							\bibitem{Khardani} Khardani, S. (2023). A Bernstein polynomial approach to the estimation of a distribution function and quantiles under censorship model. {\it Communications in Statistics - Theory and Methods}, 1--14.
							\bibitem{Ahmed} Ahmed, H. M. (2023). Numerical solutions of high-order differential equations with polynomial coefficients using a Bernstein polynomial basis. {\it Mediterranean Journal of Mathematics}, {\bf 20}: 303.
							\bibitem{Costabile} Costabile, F., Gualtieri, M. I. and Serra, S. (1996). Asymptotic expansion and extrapolation for Bernstein polynomials with applications. {\it BIT Numerical Mathematics}, {\bf 36}: 676--687.
							\bibitem{Wen} Wen, G. and Liu, Y. (2023). Tracking control based on adaptive Bernstein polynomial approximation for a class of unknown nonlinear dynamic systems. {\it Journal of the Franklin Institute}, {\bf 360}: 5082--5091.
							\bibitem{Occorsio} Occorsio, D. (2011). Some new properties of Generalized Bernstein polynomials. {\it Stud. Univ. Babes-Bolyai Math}, {\bf 56}: 147--160.
							\bibitem{Duchon} Duchon, M. (2012). A generalized Bernstein approximation theorem. {\it Tatra Mountains Mathematical Publications}, {\bf 49}: 99--109.
							\bibitem{Foupouagnigni} Foupouagnigni, M. and Mouafo Wouodji{\'e}, M. (2020). On multivariate Bernstein polynomials. {\it Mathematics}, {\bf 8}: 1397.
							\bibitem{Ri} Ri, S. (2017). A new nonlinear fractal interpolation function. {\it Fractals}, {\bf 25}: 1750063.
							\bibitem{Kim} Kim, J., Kim, H. and Mun, H. (2020). Nonlinear Fractal Interpolation Curves with Function Vertical Scaling Factors. {\it Indian. J. Pure. Appl. Math.}, {\bf 51}: 483--499.
							\bibitem{Kobes} Kobes, R. and Penner, A. J. (2005).  Nonlinear Fractal Interpolating Functions of One and Two Variables. {\it Fractals}, {\bf 13}: 179--186.
							\bibitem{Massopust}  Massopust, P. R. (1990).  Fractal surfaces. {\it J. Math. Anal. Appl.}, {\bf 151}: 275--290.
							\bibitem{Geronimo}  Geronimo, J. S. and Hardin, D. (1993). Fractal interpolation surfaces and a related 2-D multiresolution analysis. {\it J. Math. Anal. Appl.}, {\bf 176}: 561--586.
							\bibitem{Dalla} Dalla, L. (2002). Bivariate fractal interpolation functions on grids. {\it Fractals}, {\bf 10}: 53--58.
							\bibitem{Drakopoulos} Drakopoulos, V. and Manousopoulos, P. (2020). On Non-tensor product bivariate fractal interpolation surfaces on rectangular grids. {\it Mathematics}, {\bf 8}: 525.
							\bibitem{Ruan} Ruan, H. and Xu, Q. (2015). Fractal interpolation surfaces on rectangular grids. {\it Bulletin of the Australian Mathematical Society}, {\bf 91}: 435--446.
							
							\bibitem{Aparna} Aparna, M. P. and Paramanathan, P. (2023). Bivariate fractal interpolation functions on triangular domain for numerical integration and approximation. {\it International Journal of Computational Methods}, 2350019. https://doi.org/10.1142/S0219876223500196.
							
							\bibitem{Vijender} Vijender, N. and Drakopoulos, V. (2020). On the Bernstein Affine Fractal Interpolation Curved Lines and Surfaces. {\it Axioms}, {\bf 9}: 119.
						\end{thebibliography}
\end{document}